\documentclass[leqno]{article}

\usepackage[frenchb,english]{babel}
\usepackage[utf8]{inputenc}
\usepackage{amsmath}
\usepackage{amssymb}
\usepackage{amsfonts}
\usepackage{enumerate}
\usepackage{vmargin}
\usepackage[all]{xy}
\usepackage{mathrsfs}
\usepackage{mathtools}
\usepackage{lmodern}
\usepackage{slashed}
\usepackage[colorlinks=true,linkcolor=blue,pagebackref=true]{hyperref}%
\setmarginsrb{3cm}{3cm}{3.5cm}{3cm}{0cm}{0cm}{1.5cm}{3cm}
\usepackage{comment}

\usepackage{accents}

\newlength{\dhatheight}

\newcommand{\A}{\ensuremath{\mathcal{A}}}

\newcommand{\B}{\mathrm{B}} 
\let\H\relax 
\newcommand{\H}{\mathrm{H}}
\newcommand{\I}{\mathrm{I}} 
\newcommand{\K}{\ensuremath{\mathbb{K}}}

\let\L\relax 
\newcommand{\L}{\mathrm{L}}
\let\O\relax 
\newcommand{\O}{\ensuremath{\mathbb{O}}}

\newcommand{\IV}{\mathrm{IV}} 
\newcommand{\VI}{\mathrm{VI}} 

\newcommand{\M}{\mathrm{M}}

\newcommand{\absco}{\mathrm{absco}} 

\let\cal\relax
\newcommand{\cal}{\mathcal}

\newcommand{\Zc}{\mathrm{Z}}
\newcommand{\R}{\ensuremath{\mathbb{R}}}

\newcommand{\Id}{\mathrm{Id}}
\newcommand{\Sym}{\mathrm{Sym}} 
\newcommand{\la}{\langle}
\newcommand{\ra}{\rangle}

\renewcommand{\leq}{\ensuremath{\leqslant}}
\renewcommand{\geq}{\ensuremath{\geqslant}}
\newcommand{\qed}{\hfill \vrule height6pt  width6pt depth0pt}
\newcommand{\bnorm}[1]{ \big\| #1  \big\|}

\newcommand{\norm}[1]{\left\Vert#1\right\Vert}

\newcommand{\co}{\colon}

\newcommand{\ot}{\otimes}
\newcommand{\ovl}{\overline}

\let\i\relax 
\newcommand{\i}{\mathrm{i}}
\let\j\relax 
\newcommand{\j}{\mathrm{j}} 
\let\k\relax 
\newcommand{\k}{\mathrm{k}} 

\newcommand{\ov}{\overset}
\newcommand{\sa}{\mathrm{sa}}
\newcommand{\JW}{\mathrm{JW}}

\newcommand{\JB}{\mathrm{JB}}

\newcommand{\JBW}{\mathrm{JBW}}

\renewcommand{\d}{\mathop{}\mathopen{}\mathrm{d}} 
\renewcommand{\d}{\mathop{}\mathopen{}\mathrm{d}}

\DeclareMathOperator{\Span}{span} 
\DeclareMathOperator{\tr}{Tr} 
\DeclareMathOperator{\Tr}{Tr} 
\DeclareMathOperator{\Ran}{Ran} 
\DeclareMathOperator{\dom}{dom} 
\let\Re\relax 
\DeclareMathOperator{\Re}{Re} 
\let\Im\relax 
\DeclareMathOperator{\Im}{Im} 
\DeclareMathOperator{\card}{card} 

\newtheorem{thm}{Theorem}[section]
\newtheorem{defi}[thm]{Definition}
\newtheorem{prop}[thm]{Proposition}

\newtheorem{quest}[thm]{Question}

\newtheorem{cor}[thm]{Corollary}

\newtheorem{remark}[thm]{Remark}
\newtheorem{example}[thm]{Example}

\newenvironment{proof}[1][]{\noindent {\it Proof #1} : }{\hbox{~}\qed
\smallskip
}

\usepackage{tocloft}
\numberwithin{equation}{section}
\usepackage[nottoc,notlot,notlof]{tocbibind}

\let\OLDthebibliography\thebibliography
\renewcommand\thebibliography[1]{
  \OLDthebibliography{#1}
  \setlength{\parskip}{0pt}
  \setlength{\itemsep}{0pt plus 0.3ex}
}

\newcommand\reallywidehat[1]{\arraycolsep=0pt\relax%
\begin{array}{c}
\stretchto{
  \scaleto{
    \scalerel*[\widthof{\ensuremath{#1}}]{\kern-.5pt\bigwedge\kern-.5pt}
    {\rule[-\textheight/2]{1ex}{\textheight}} 
  }{\textheight} %
}{0.5ex}\\           
#1\\                 
\rule{-1ex}{0ex}
\end{array}
}

\begin{document}
\selectlanguage{english}
\title{\bfseries{Spectral versus interpolation norms in tracial nonassociative $\L^p$-spaces}}
\date{}
\author{\bfseries{C\'edric Arhancet and Lei Li}}%
\maketitle


\begin{abstract}
We investigate the metric structure of nonassociative $\mathrm{L}^p$-spaces associated with tracial $\mathrm{JW}^*$-algebras. While noncommutative $\mathrm{L}^p$-spaces arising from von Neumann algebras enjoy a unique natural norm, the situation in the Jordan setting is more subtle. We compare two canonical definitions: the interpolation norm, arising from the complex method between the algebra and its predual, and the spectral norm, defined with the trace. We show that these two norms are equivalent but generally not isometric for $p \neq 2$, even in the associative case of nonabelian von Neumann algebras when viewed through the Jordan product, thereby answering an open question raised by the first author in a previous paper. We further analyze the geometry of these spaces in concrete examples such  as complex spin factors or the complexified Albert algebra. Finally, we discuss the relevance of these results to generalized probabilistic theories (GPTs), where Jordan structures arise naturally, and explain why $\JBW$-algebras and their preduals provide a natural framework for such models.
\end{abstract}
%
%
%
%
%
%

\makeatletter
 \renewcommand{\@makefntext}[1]{#1}
 \makeatother

{
  \hypersetup{linkcolor=blue}
 \tableofcontents
}

\section{Introduction}
\label{sec:Introduction}

Noncommutative $\L^p$-spaces associated with von Neumann algebras have become central objects in modern functional analysis, with deep connections to quantum information theory, harmonic analysis, and operator space theory (see e.g.~\cite{Arh25}, \cite{Arh26}, \cite{PiX03}, \cite{Pis98}). These spaces extend the classical Lebesgue integration theory to the setting of associative algebras of operators. 

However, the foundational works \cite{Jor33} and \cite{JNW34} of Jordan, von Neumann, and Wigner suggest that the most general framework for quantum mechanics relies not on associative algebras, but on \textit{Jordan algebras}. Indeed, in his seminal paper \cite[(4)]{Jor33}, Jordan introduced a novel algebraic framework designed to capture the statistical properties of quantum mechanical observables. Recognizing that the standard associative product of two Hermitian matrices (representing physical observables) is not necessarily Hermitian, Jordan proposed shifting the focus to the symmetrized product:
\begin{equation}
\label{Jordan-product}
x \circ y 
\ov{\mathrm{def}}{=} \frac{1}{2}(xy+yx).
\end{equation}
This construction ensures that the resulting structure remains within the space of observables. Jordan demonstrated in \cite[(12) p.~288]{JNW34} that while this product is nonassociative, it satisfies a weak form of associativity, now known as the Jordan identity:
\begin{equation}
\label{Jordan-identity}
x \circ (y \circ x^2) 
=(x \circ y) \circ x^2.
\end{equation}
This work laid the mathematical groundwork for what are now called Jordan algebras, providing a rigorous alternative to the standard associative operator theory. Building upon this initial discovery, the collaborative work \cite{JNW34} by Jordan, von Neumann, and Wigner sought to determine whether this algebraic approach could offer a more general foundation for quantum mechanics than the standard Hilbert space formalism. Their exhaustive classification of formally real\footnote{\thefootnote. A Jordan algebra $\A$ is called formally real if $a_1 \circ a_1 + \cdots +a_n \circ a_n =0$ implies $a_1 = \cdots = a_n = 0$ for any $a_1, \ldots, a_n \in \A$.} Jordan algebras of finite dimension revealed that almost all such structures are representable by classical matrices, with the notable exception of the exceptional Jordan algebra, thereby delimiting the algebraic boundaries of quantum theory.

Nowadays, these structures allow for the description of some Generalized Probabilistic Theories (GPTs) \cite{Pla23} \cite{SAH25}, including real quantum mechanics and exceptional systems based on octonions, which are currently attracting renewed interest in the study of quantum foundations.

In finite dimension, formally real Jordan algebras coincide with finite-dimensional $\JBW$-algebras. Thus, the Jordan algebras arising in the work of Jordan, von Neumann, and Wigner already fit into the broader framework of $\JBW$-algebras and, upon complexification, into that of $\JBW^*$-algebras. The latter class, introduced by Edwards in \cite{Edw80}, also contains the exceptional $\JBW^*$-factor $\mathrm{H}_3(\O_\mathbb{C})$ consisting of Hermitian $3 \times 3$ matrices over the algebra $\O_\mathbb{C}$ of complex octonions. Recall that a $\JW^*$-algebra is a weak*-closed Jordan-$*$-subalgebra of a von Neumann algebra. It is a particular case of a $\JBW^*$-algebra. Observe that any von Neumann algebra $\cal{M}$ becomes a $\JW^*$-algebra when endowed with the Jordan product \eqref{Jordan-product}. Consequently, every von Neumann algebra provides an example of a $\JW^*$-algebra and therefore of a $\JBW^*$-algebra.

It is known that the selfadjoint parts of $\JBW^*$-algebras and $\JW^*$-algebras coincide with $\JBW$-algebras and $\JW$-algebras, respectively. These are real linear spaces of selfadjoint operators that are weak* closed and stable under the Jordan product $\circ$. The theory of $\JW$-algebras was introduced by Topping in \cite{Top65} and further developed by several authors. For further background on these structures, we refer to the books \cite{AlS03}, \cite{ARU97}, \cite{CGRP14}, \cite{CGRP18}, \cite{HOS84}, as well as to the important paper \cite{Sto66} and the references therein. Finally, in this context, we can define the notion of trace on a $\JBW^*$-algebra which is a linear functional $\tau$ on $\cal{M}$ satisfying the equality
\begin{equation}
\label{associative-trace}
\tau (x \circ(y \circ z))
=\tau((x \circ y) \circ z), \quad x,y,z \in \cal{M}.
\end{equation}

In a previous work \cite{Arh24a}, we initiated the study of nonassociative $\L^p$-spaces associated with general $\JBW^*$-algebras. A fundamental question left open was the precise relationship between the analytical and spectral definitions of these spaces in the case where a normal finite faithful trace is available. Indeed, in the tracial setting, two approaches naturally present themselves. The first approach consists of defining the Banach space $\L^{p,A}(\mathcal{M})$ by complex interpolation \cite{BeL76,KPS82} between a $\JBW^*$-algebra $\cal{M}$, equipped with the operator norm, its product $\circ$ and a (normal finite faithful) trace $\tau$, and its predual $\cal{M}_*$, via the embedding $i \co \cal{M} \to \cal{M}_*$, $x \mapsto (y \mapsto \tau(x \circ y ))$ by letting
\begin{equation}
\label{Def-Lp-state-JBW-intro}
\L^{p,A}(\cal{M})
\ov{\mathrm{def}}{=} (\cal{M},\cal{M}_*)_{\frac{1}{p}},
\end{equation}
for any $1 <p < \infty$. We let $\L^\infty(\cal{M}) \ov{\mathrm{def}}{=} \cal{M}$. We also refer to \cite{Arh24a} for a generalization to $\JBW^*$-algebras equipped with normal faithful states and a description of the links with the spaces of \cite{Ioc86} associated with $\JBW$-algebras and to \cite[Section 8]{Arh24b} for a generalization of this approach to $\JBW^*$-triples equipped with a weak* continuous linear functional such that the support tripotent is complete. 

Note that in the particular case where the $\JBW^*$-algebra $\cal{M}$ is a von Neumann algebra, equipped with the Jordan product \eqref{Jordan-product} and a normal finite faithful trace $\tau$, the previous embedding is precisely $i \co \cal{M} \to \cal{M}_*$, $x \mapsto (y \mapsto \tau(x y))$. This embedding is used in the paper \cite{Kos84}. Consequently, the Banach space $\L^{p,A}(\cal{M})$ coincides with the Kosaki noncommutative $\L^p$-space $\L^{p,K}(\cal{M})$ introduced in \cite{Kos84}, i.e.
\begin{equation}
\label{arh=Dix-if-trace}
\L^{p,A}(\cal{M})
= \L^{p,K}(\cal{M}).
\end{equation}
It is well-known \cite{PiX03} that the Kosaki noncommutative $\L^p$-space $\L^{p,K}(\cal{M})$ is isometrically isomorphic to the Dixmier noncommutative $\L^p$-space $\L^{p,D}(\cal{M})$. This means that we have an isometry
\begin{equation}
\label{Kosaki-vs-Dixmier}
\L^{p,K}(\cal{M})
\cong \L^{p,D}(\cal{M}).
\end{equation}
This latter space is introduced in \cite{Dix53} and defined as the completion of $\cal{M}$ for the norm
\begin{equation}
\label{Dixmier-NC-spaces}
\norm{x}_{\L^{p,D}(\cal{M})} 
\ov{\mathrm{def}}{=} \tau\big(|x|^p\big)^{\frac{1}{p}}, \quad x \in \cal{M},
\end{equation}
where we use the modulus $|x| \ov{\mathrm{def}}{=} (x^*x)^{\frac{1}{2}}$.

The second approach for defining nonassociative $\L^p$-spaces associated with a $\JBW^*$-algebra $\cal{M}$ endowed with a normal finite faithful trace $\tau$, is inspired by the previous approach of Dixmier and consists of introducing the notation
\begin{equation}
\label{norm-LpNA}
\norm{x}_{\L^p(\cal{M})}
\ov{\mathrm{def}}{=} \big(\tau \big[(x^* \circ x)^{\frac{p}{2}}\big]\big)^{\frac{1}{p}}, \quad x \in \cal{M},
\end{equation}
using the trace of the $p$-th power of the Jordan modulus $(x^* \circ x)^{\frac{1}{2}}$. It is established in \cite[Theorem 3.15]{Arh24a} that $\norm{\cdot}_{\L^{p,A}(\cal{M})}$ and $\norm{\cdot}_{\L^p(\cal{M})}$ coincide on the \textit{selfadjoint} part $\cal{M}_\sa$ of $\cal{M}$. In this paper, if $\cal{M}$ is a $\JW^*$-algebra we show that this formula defines a norm on $\cal{M}$ if $1 \leq p < \infty$, answering the first part of the open question \cite[Question 5.4]{Arh24a}. Thus, we can introduce the completion $\L^p(\cal{M})$ of $\cal{M}$.

In this paper, we show in Theorem \ref{Th-comparison-interpolation} that the \textit{isometric} coincidence \eqref{Kosaki-vs-Dixmier} breaks down in the general nonassociative setting, i.e.~the Banach spaces $\L^p(\cal{M})$ and $\L^{p,A}(\cal{M})$ are not isometric, answering the second part of \cite[Question 5.4]{Arh24a}. Indeed, by investigating the Jordan structure underlying von Neumann algebras, we prove that the spectral norm $\norm{\cdot}_{\L^p(\cal{M})}$ defined in \eqref{norm-LpNA} and the interpolation norm $\norm{\cdot}_{\L^{p,A}(\cal{M})}$ are \textit{equivalent but distinct}: we have a generally non-isometric isomorphism
\begin{equation}
\label{isomorphism-intro}
\L^p(\cal{M}) 
\approx \L^{p,A}(\cal{M}),
\end{equation}
for any $1 < p  < \infty$.

Our main result (Theorem \ref{Th-comparison-interpolation}) establishes that the spectral norm dominates the interpolation norm, with an optimal constant that we compute. This phenomenon highlights a fundamental geometric rigidity cost imposed by the symmetrized Jordan product \eqref{Jordan-product}.

We illustrate these phenomena through detailed computations on the building blocks of nonassociative algebra theory. In particular, we describe nonassociative $\L^p$-spaces of spin systems. We also study the $\L^p$-spaces associated with the algebra $\mathrm{H}_3(\mathbb{O}_{\mathbb{C}})$, providing the necessary analytical framework for this structure which cannot be embedded as a weak* closed Jordan-$*$-subalgebra into any $\JW^*$-algebra.

Our motivation is twofold. First, we aim to gain a better understanding of contractively complemented subspaces of noncommutative $\L^p$-spaces arising as the ranges of contractive projections, where the interpolation spaces $\L^{p,A}(\cal{M})$ defined in \eqref{Def-Lp-state-JBW-intro} appear. It is proved in an unpublished extension of the preprint \cite{Arh26b} that a complemented subspace of a Dixmier noncommutative $\L^p$-space associated with a finite von Neumann algebra arising as the range of a \textit{positive} contractive projection is isometrically isomorphic to a nonassociative $\L^{p}$-space $\L^{p,A}(\cal{M})$. We refer to \cite{ArF92}, \cite{ArR24}, \cite{Arh24a}, \cite{Arh24b}, \cite{Arh26b}, \cite{Bof25} and \cite{LRR09} for the state of the art on this topic.

Second, the Banach spaces $\L^p(\cal{M})$ and $\L^{p,A}(\cal{M})$ naturally fit within the framework of generalized probabilistic theories. While noncommutative $\L^p$-spaces are foundational to standard Quantum Information Theory, nonassociative $\L^p$-spaces provide the necessary analytical framework for some \textit{suitable} generalized probabilistic theories, encompassing complex Quantum Mechanics, exceptional systems based on the Albert algebra and  spin-factor models. Specifically, our results provide the rigorous mathematical grounds to define and study entropies in systems where associativity is lost. For further information on this theory, we refer the reader to \cite{BGW15}, \cite{Bar07}, \cite{BaW11}, \cite{BaW14}, \cite{BGW20}, \cite{BUW23}, \cite{CDP11a}, \cite{CDP11b}, \cite{CFR01}, \cite{Har01}, \cite{HaW12}, \cite{JaH14}, \cite{LRTF21}, \cite{LPW18}, \cite{Mul21}, \cite{Nie20}, \cite{Pla23}, \cite{SAH25}, \cite{Sha21}, \cite{Wil19}, \cite{Wil25}, \cite{WW21} and the references therein.

We also refer to \cite{CGRP18}, \cite{EPV25b}, \cite{Isi19}, \cite{McC78}, \cite{McC04}, \cite{Upm85} and \cite{Upm87} for a description of some applications of Jordan algebras to complex analysis in finite and infinite dimensions, harmonic analysis on homogeneous spaces, operator theory, projective geometry and foundations of quantum mechanics.

\paragraph{Structure of the paper}
In Section \ref{sec-preliminaries}, we recall basic notions from Jordan algebra theory and complex interpolation theory.  In Section \ref{sec-norm}, we prove that the formula \eqref{norm-LpNA} defines a norm on every $\JW^*$-algebra $\cal{M}$ equipped with a normal finite faithful trace $\tau$. In Section \ref{sec-VNA}, we establish our main isomorphism result in the associative setting of von Neumann algebras. Section \ref{sec-isomorphism} is devoted to extending this result to some $\JW^*$-algebras. In Section \ref{sec-duality}, we show that the duality relation $(\L^1)^*=\L^\infty$ fails for the spaces $\L^p(\cal{M})$ defined via \eqref{norm-LpNA}. In Section \ref{sec-Complex-spin-factors}, we investigate complex spin factors. In Section \ref{sec-Generalized}, we explain how our results fit naturally into the framework of generalized probabilistic theories. Finally, in Section \ref{sec-open-question}, we state some open questions.



\section{Preliminaries}
\label{sec-preliminaries}

The purpose of this section is twofold: we fix terminology from Jordan operator algebra theory, and we isolate the few structural facts that will be repeatedly used in the paper.

\paragraph{Jordan algebras} A Jordan algebra $\A$ over a field $\K$ is a vector space $\A$ over $\K$ equipped with a (not necessarily associative) commutative bilinear product $\circ$ that satisfies the Jordan identity $x \circ (x^2 \circ y) =x^2\circ (x \circ y)$ of \eqref{Jordan-identity} for any $x,y \in \A$, see e.g.~\cite[p.~162]{CGRP14} or \cite[Definition 1.1 p.~3]{AlS03}. This means that the multiplication operators by $x$ and $x^2$ commute. 


Following \cite[Definition 1.5 p.~5]{AlS03} and \cite[3.1.4 p.~76]{HOS84}, a $\JB$-algebra is a Jordan algebra over the scalar field $\R$ equipped with a complete norm satisfying the properties 
\begin{equation}
\label{def-JB}
\norm{x \circ y} \leq \norm{x} \norm{y}, \quad \|x^2\|=\norm{x}^2 
\quad \text{and} \quad \|x^2\| \leq \|x^2+y^2\|
\end{equation}
for any $x,y \in \A$. A $\JBW$-algebra is a $\JB$-algebra which is a dual Banach space \cite[p.~111]{HOS84}. In this case, the predual is unique.

A $\JB^*$-algebra \cite[p.~91]{HOS84} \cite[Definition 3.3.1 p.~345]{CGRP14} is a complex Banach space $\cal{M}$ which is a complex Jordan algebra equipped with an involution satisfying 
\begin{equation}
\label{def-JBstar}
\norm{x \circ y} \leq \norm{x} \norm{y},\quad  \norm{x^*}=\norm{x} 
\quad \text{and} \quad \norm{\{x,x^*,x\}}=\norm{x}^3
\end{equation}
for any $x,y \in \cal{M}$, where we use the Jordan triple product \cite[p.~25]{HOS84}\footnote{\thefootnote. We warn the reader that another definition is often used in some papers for the triple product:
\begin{equation}
\label{triple product} 
\{x,y,z\} 
 \ov{\mathrm{def}}{=} (x \circ y^*) \circ z + (y^* \circ z) \circ x - (x \circ z) \circ y^*.
\end{equation}}  
$$
\{x,y,z\}
\ov{\mathrm{def}}{=} (x \circ y) \circ z+(y \circ z) \circ x-(x \circ z) \circ y.
$$ 
As in \cite[p.~4]{CGRP18}, we say that a $\JB^*$-algebra which is a dual Banach space is a $\JBW^*$-algebra. By \cite[Lemma 4.1.7 p.~96]{HOS84}, any $\JBW$-algebra is unital.

An element $x$ in a $\JB^*$-algebra is called selfadjoint if $x^* = x$. According to \cite[Corollary 5.1.29 p.~9]{CGRP18}, the set $\cal{M}_\sa$ of selfadjoint elements of a $\JBW^*$-algebra $\cal{M}$ has a canonical structure of $\JBW$-algebra. Conversely, if $\A$ is a $\JBW$-algebra then by \cite[Corollary 5.1.41 p.~15]{CGRP18} there exists a unique $\JBW^*$-algebra $\cal{M}$ such that $\A$ is the selfadjoint part $\cal{M}_\sa$ of $\cal{M}$. 


An element $x$ of a $\JB^*$-algebra $\cal{M}$ is said to be positive \cite[p.~9]{CGRP18} \cite[p.~7]{AlS03} if $x \in \cal{M}_\sa$ and if we can write $x = y \circ y$ for some element $y \in \cal{M}_\sa$.

%


\paragraph{Centers and factors} Two elements $x$ and $y$ of a Jordan algebra $\A$ are said to operator commute \cite[p.~44]{HOS84} 
if for any $z \in \A$ we have $(x \circ z) \circ y= x \circ (z \circ y)$. The center $\Zc(\A)$ of $\A$ \cite[p.~45]{HOS84}  is the set of all elements of $\A$ which operator commute with all elements of $\A$. We say that an element $x \in \A$ is central if it belongs to the center. By \cite[Lemma 2.5.3 p.~45]{HOS84} and \cite[Proposition 2.36 p.~56]{AlS03}, the center $\Zc(\A)$ is an associative $\JBW$-subalgebra of $\A$. 

Following \cite[p.~115]{HOS84} and \cite[p.~348]{CGRP18}, if the center of a non-zero $\JBW$-algebra $\A$ only consists of scalar multiples of the identity, we say that $\A$ is a $\JBW$-factor. In this case, we say that the associated $\JBW^*$-algebra $\cal{M}$, such that $\cal{M}_\sa=\A$, is a $\JBW^*$-factor. We refer to \cite[Theorem 6.1.40 p.~362]{CGRP18} for a classification of $\JBW$-factors and $\JBW^*$-factors.

\paragraph{$\JW$-algebras} Recall that a (concrete) $\JW$-algebra \cite[p.~95]{HOS84} \cite[Definition 2.70 p.~70]{AlS03} 
 is a weak* closed Jordan subalgebra $\A$ of the selfadjoint part $\B(H)_\sa$ of the space $\B(H)$ of bounded operators on some complex Hilbert space $H$, that is a real linear space of selfadjoint operators which is closed for the weak* topology and closed under the Jordan product $\circ$. Note that a $\JW$-algebra is a $\JBW$-algebra by \cite[p.~95]{HOS84}. In this situation, by \cite[Proposition 1.49 p.~28]{AlS03}, two elements $x$ and $y$ of a $\JW$-algebra $\A$ operator commute if and only if $x$ and $y$ commute in $\B(H)$. A $\JW$-algebra which is a $\JBW$-factor is called a $\JW$-factor.

\paragraph{Real spin factors} Real spin factors \cite{Top65} are among the basic examples in Jordan theory. They will later serve as test cases where computations can be carried out explicitly.

\begin{example} \normalfont
\label{spin-factor}
Let $\cal{H}$ be a real Hilbert space, and let $\R 1$ denote a one-dimensional real Hilbert space with unit vector 1. Let $\A \ov{\mathrm{def}}{=} \cal{H} \oplus \R 1$ and consider the product $\circ$ on $\A$ defined by
\begin{equation}
\label{product-spin-factor}
(a+\lambda 1)  \circ  (b  + \mu 1)  
\ov{\mathrm{def}}{=} \mu a  + \lambda b  +  (\la a ,  b\ra_{\cal{H}}  +  \lambda \mu)1, \quad a,b \in \cal{H}, \lambda,\mu \in \R
\end{equation}
and the norm $\norm{a+\lambda 1}_\A  \ov{\mathrm{def}}{=} \norm{a}_\cal{H}  + |\lambda|$. According to \cite[Corollary 30 p.~42]{Top65}, $\A$ is a $\JBW$-algebra. If the dimension of the real Hilbert space $\cal{H}$ is greater than or equal to $2$, $\A$ is a $\JBW$-factor by \cite[Proposition 3.37 p.~92]{AlS03}. It is called a (abstract) real spin factor and admits a concrete representation as a $\JW$-algebra by \cite[Theorem 4.1 p.~103]{AlS03}. These algebras were introduced by Topping in \cite{Top65} and \cite{Top66}. He showed in \cite[Theorem 3 p.~1061]{Top66} (see also \cite[Proposition 6.1.5 p.~137]{HOS84}) that two real spin factors are isomorphic if and only if their real Hilbert space dimensions are equal.

We can give a concrete definition. Consider a spin system, that is a set $\cal{S}$ of symmetries (i.e.~selfadjoint unitaries) $\not=\pm\Id$ acting on some complex Hilbert space $H$ which anticommute (i.e.~$st=-ts$ for any distinct $t,s \in \cal{S}$). Then according to \cite[Theorem 1 p.~1059]{Top66}, the subspace $\ovl{\Span_\R\cal{S}} \oplus \R \Id_H$ of the space $\B(H)$ is a $\JW$-algebra which is isomorphic as a $\JBW$-algebra to some real spin factor. 


Conversely, by \cite[Theorem 2 p.~1061]{Top66} any real spin factor is isomorphic to some $\JW$-algebra associated with some spin system. We refer also to \cite[p.~175]{Pis03} for additional information on spin systems.
\end{example}

\paragraph{Projections} If $p$ is a projection (i.e.~$p \circ p=p$) of a $\JBW$-algebra $\A$, the smallest central projection $q$ such that $q \geq p$ is called the central cover of $p$ and denoted by $c(p)$, see \cite[Definition 2.38 p.~56]{AlS03}. We say that a projection $p$ of a $\JBW$-algebra $\A$ is abelian if the algebra $\{p,\A,p\}$ is associative \cite[p.~122]{HOS84} \cite[Definition 3.14 p.~85]{AlS03}. Two projections $p,q \in \A$ are said to be orthogonal if $p \circ q=0$ \cite[p.~45]{AlS03}. An element $p$ of a $\JBW^*$-algebra is said to be a projection if $p^* = p$ and $p \circ p = p$.


\paragraph{$\JW^*$-algebras} We define a $\JW^*$-algebra as a complex subspace of the space $\B(H)$ which is closed for the weak* topology and closed under the Jordan product $\circ$ and the involution, for some complex Hilbert space $H$.  
A $\JW^*$-algebra is a $\JBW^*$-algebra. The selfadjoint part of a $\JW^*$-algebra is a $\JW$-algebra. Conversely, if $\A$ is a $\JW$-algebra included in $\B(H)$ then the complexification $\A_\mathbb{C}=\A+\i \A$ is a $\JW^*$-algebra included in $\B(H)$. 


\paragraph{Traces} A trace on a $\JBW$-algebra $\A$ is a function $\tau$ defined on the subset $\A_+$ of positive
elements of $\A$ with values in $[0,+\infty]$ satisfying the following conditions:
\begin{enumerate}
\item $\tau(x+y)=\tau(x)+\tau(y)$ for any $x, y \in \A_+$,
\item $\tau(\lambda x)=\lambda\tau(x)$ for any $x \in \A_+$ and any $\lambda \geq 0$, where $0.(+\infty)=0$,
\item $\tau(s \circ x \circ s)=\tau(x)$ for any $x \in \A_+$ and any arbitrary symmetry $s$ of $\A$.
\end{enumerate}
Recall that a symmetry is an element $s$ such that $s \circ s=1$. The trace $\tau$ is said to be faithful if $\tau(x) > 0$ for all non-zero $x \in \A_+$, finite if $\tau(1) < + \infty$, semifinite if given any non-zero $x \in \A_+$ there is a non-zero $y \in \A_+$ such that $y \leq x$ with $\tau(y) < +\infty$. The trace $\tau$ is normal if for every increasing net $(x_i)$ of positive elements such that $x_i \to x$ where $x \in \A_+$, we have $\tau(x_i) \to \tau(x)$. We refer to \cite{AyA85}, \cite{Ayu82}, \cite{Ayu92}, \cite{Kin83} and \cite{PeS82} for more information on traces on $\JBW$-algebras.

Every finite trace on a $\JBW$-algebra $\A$ can be extended by linearity to a linear functional on $\A$. Thus a finite trace on a $\JBW$-algebra $\A$ can be seen as a positive linear functional $\tau$ satisfying the condition $\tau(s \circ x \circ s)=\tau(x)$ for all $x \in \A$ and all symmetries $s \in \A$. By \cite[Lemma 5.18 p.~147]{AlS03}, it is known that the last condition is equivalent to the formula 
\begin{equation}
\label{Def-trace}
\tau (x \circ (y \circ z)) 
= \tau((x \circ y) \circ z), \quad x,y,z \in \A.
\end{equation}
By complexification with \cite[Corollary 5.1.41 p.~15]{CGRP18}, we obtain a normal faithful positive linear functional on the associated $\JBW^*$-algebra $\cal{M}$ satisfying 
\begin{equation}
\label{trace}
\tau (x \circ(y \circ z))
=\tau((x \circ y) \circ z), \quad x,y,z \in \cal{M}.
\end{equation}
We say that such a map is a normal finite faithful trace on $\cal{M}$.







\begin{example} \normalfont
\label{ex-trace-spin}
Consider a real spin factor $\A = \cal{H} \oplus \R 1$ as in Example \ref{spin-factor}. By \cite[Lemma 5.21 p.~149]{AlS03} (see also \cite[Proposition 6.1.7 p.~137]{HOS84} and \cite{Top66}), there exists a unique tracial state $\tau$ on $\A$. Moreover, the same reference shows that $\tau$ is normal, faithful and defined by $\tau(1)=1$ and $\tau(a)=0$ for any $a \in \cal{H}$, i.e.
\begin{equation}
\label{trace-spin}
\tau(a+\lambda 1)=\lambda, \quad a \in \cal{H}, \lambda \in \R.
\end{equation}
\end{example}

The connection between traces on von Neumann algebras and traces on the associated $\JW^*$-algebras is clarified by \cite[Remark p.~149]{AlS03}, which says that both notions agree.

\section{Constructing the nonassociative $\L^p$-spaces of a $\JW^*$-algebra with a trace}
\label{sec-norm}

Consider some $\JW^*$-algebra $\cal{M}$ equipped with a normal finite faithful trace $\tau$. Assume that $1 \leq p < \infty$. In this section, we establish that the formula $\norm{x}_{\L^p(\cal{M})}
\ov{\mathrm{def}}{=} \big(\tau \big[(x^* \circ x)^{\frac{p}{2}}\big]\big)^{\frac{1}{p}}$ (where $x \in \cal{M}$) of \eqref{norm-LpNA} defines a norm on $\cal{M}$. 

\begin{prop} 
\label{prop:norm-homo}
Let $\cal{M}$ be a $\JBW^*$-algebra equipped with a normal finite faithful trace $\tau$. Suppose that $1 \leq p < \infty$. The expression of \eqref{norm-LpNA} is absolutely homogeneous and point-separating.
\end{prop}

\begin{proof}
By \cite[(5.1.2) p.~9]{CGRP18} or \cite[Lemma 3.1 (ii) p.~7]{HKPP20}, we have $x \circ x^* \geq 0$ for any $x \in \cal{M}$. The homogeneity follows from the computation 
\begin{align*}
\MoveEqLeft
\norm{\lambda x}_{\L^p(\cal{M})}
\ov{\eqref{norm-LpNA}}{=} \big[\tau\big(((\lambda x)^* \circ (\lambda x))^{\frac{p}{2}}\big) \big]^{\frac{1}{p}}
= \big[\tau\big((|\lambda|^2 x^* \circ x)^{\frac{p}{2}}\big) \big]^{\frac{1}{p}} \\
&= |\lambda| \big[\tau\big((x^* \circ x)^{\frac{p}{2}}\big) \big]^{\frac{1}{p}}
\ov{\eqref{norm-LpNA}}{=}  |\lambda| \norm{x}_{\L^p(\cal{M})}.
\end{align*}
Suppose that $\norm{x}_{\L^p(\cal{M})} = 0$ for some $x \in \cal{M}$. We have $\tau\big[(x^* \circ x)^{\frac{p}{2}}\big]=0$. By the faithfulness of the trace $\tau$, we infer that $(x^* \circ x)^{\frac{p}{2}}=0$. Hence $x \circ x^*=0$. By \cite[Lemma 3.1 (i) p.~7]{HKPP20} or \cite[Lemma 3.4.65 p.~382]{CGRP14}, we deduce that $x=0$. 
\end{proof}

The non-trivial part is the triangle inequality. We start with the case of $\JW^*$-algebras associated with von Neumann algebras. In this setting, recall that $\circ$ denotes the Jordan product introduced in \eqref{Jordan-product}. Our approach is to embed the Jordan modulus into an associative environment by a simple $2 \times2$ matrix argument. More precisely, we construct a canonical linear map $\Psi \co \cal{M} \to \M_2(\cal{M})$ such that the operator modulus $|\Psi(x)|$ encodes $(x^* \circ x)^{\frac{1}{2}}$. This allows us to transfer the triangle inequality from the Dixmier noncommutative $\L^p$-space of $\M_2(\cal{M})$ back to $\cal{M}$.

\begin{prop} 
\label{prop:norm-well-defined}
Let $\cal{M}$ be a von Neumann algebra equipped with a normal finite faithful trace $\tau$. Suppose that $1 \leq p < \infty$. The expression of \eqref{norm-LpNA} defines a norm on $\cal{M}$.
\end{prop}

\begin{proof}
By Proposition \ref{prop:norm-homo}, it suffices to prove the triangle inequality. Consider the matrix algebra $\M_{2}(\cal{M})$ of $2 \times 2$ matrices with entries in the von Neumann algebra $\cal{M}$. We define the injective $\R$-linear map $\Psi \co \mathcal{M} \to \M_{2}(\cal{M})$ by
\begin{equation}
\label{inter-42}
\Psi(x) 
\ov{\mathrm{def}}{=} 
\begin{bmatrix} 
x &0\\ 
x^* &0\end{bmatrix}.
\end{equation}
For any $x \in \mathcal{M}$, we compute the modulus square of $\Psi(x)$:
\begin{equation*}
|\Psi(x)|^2 
= \Psi(x)^* \Psi(x) 
\ov{\eqref{inter-42}}{=} \begin{bmatrix} 
x^* & x \\
0 & 0
\end{bmatrix} 
\begin{bmatrix} 
x & 0\\ 
x^* & 0
\end{bmatrix} 
=\begin{bmatrix}
  x^* x + x x^*   & 0  \\
    0 &  0 \\
\end{bmatrix}
\ov{\eqref{Jordan-product}}{=} 2\begin{bmatrix}
x^* \circ x   & 0  \\
    0 &  0 \\
\end{bmatrix}.
\end{equation*}
Hence
\begin{equation}
\label{inter-35}
|\Psi(x)|
=\sqrt{2}\begin{bmatrix}
(x^* \circ x)^{\frac{1}{2}}   & 0  \\
    0 &  0 \\
\end{bmatrix}.
\end{equation}
The von Neumann algebra $\M_2(\cal{M})$ is canonically equipped with the normal finite faithful trace $\tr \ot \tau$. For any $x \in \mathcal{M}$, we have
\begin{align*}
\MoveEqLeft
\norm{\Psi(x)}_{\L^{p,D}(\M_2(\cal{M}))}^p        
\ov{\eqref{Dixmier-NC-spaces}}{=}(\tr \ot \tau)(|\Psi(x)|^p)  
\ov{\eqref{inter-35}}{=} 2^{\frac{p}{2}}(\tr \ot \tau)\bigg(\begin{bmatrix}
(x^* \circ x)^{\frac{p}{2}}   & 0  \\
    0 &  0 \\
\end{bmatrix}\bigg)\\
&= 2^{\frac{p}{2}} \tau \big[(x^* \circ x)^{\frac{p}{2}}\big]
\ov{\eqref{norm-LpNA}}{=} 2^{\frac{p}{2}} \norm{x}_{\L^p(\cal{M})}^p.
\end{align*}
We conclude that
\begin{equation}
\label{inter-789}
\norm{\Psi(x)}_{\L^{p,D}(\M_2(\cal{M}))}   
= \sqrt{2} \norm{x}_{\L^p(\cal{M})}.
\end{equation}
Hence, using the triangle inequality for the noncommutative $\L^p$-space $\L^{p,D}(\M_2(\cal{M}))$, we deduce that for any $x,y \in \cal{M}$ we have
\begin{align*}
\MoveEqLeft
\norm{x+y}_{\L^p(\cal{M})}
\ov{\eqref{inter-789}}{=} \frac{1}{\sqrt{2}}\norm{\Psi(x+y)}_{\L^{p,D}(\M_2(\cal{M}))} 
= \frac{1}{\sqrt{2}}\norm{\Psi(x)+\Psi(y)}_{\L^{p,D}(\M_2(\cal{M}))} \\
&\leq \frac{1}{\sqrt{2}}\big(\norm{\Psi(x)}_{\L^{p,D}(\M_2(\cal{M}))} + \norm{\Psi(y)}_{\L^{p,D}(\M_2(\cal{M}))}\big) 
\ov{\eqref{inter-789}}{=} \norm{x}_{\L^p(\cal{M})} + \norm{y}_{\L^p(\cal{M})}.
\end{align*}
\end{proof}

Since the von Neumann algebra case is established, the general $\JW^*$-case follows by structural reductions (reversible part, type $\I_2$ part) and by a direct sum argument based on central projections. Recall that a $\JW$-algebra $\A$ is said to be reversible \cite[Definition 4.24 p.~113]{AlS03} \cite[p.~25]{HOS84} if it is closed under symmetric products, i.e.~if $a_1,\ldots,a_k \in \A$ then 
$$
a_1a_2 \cdots a_k+a_k\cdots a_2a_1 \in \A.
$$
Note that this notion depends on the particular realization of the $\JBW$-algebra $\A$ as a weak* closed Jordan subalgebra of some $\B(H)_\sa$. It is worth noting that if $\cal{M}$ is a von Neumann algebra, then its selfadjoint part $\cal{M}_\sa \ov{\mathrm{def}}{=} \{x \in \cal{M} : x^* = x\}$ is a reversible $\JW$-algebra. Using the enveloping von Neumann algebra $\A''$ of a reversible $\JW$-algebra $\A$, we deduce the following result.

\begin{cor} 
\label{cor:norm-well-defined-JW}
Let $\cal{M}$ be a $\JW^*$-algebra such that its associated $\JW$-algebra $\A$ is reversible, equipped with a normal finite faithful trace $\tau$. Suppose that $1 \leq p < \infty$. Then the trace $\tau$ extends to a normal finite faithful trace on $\A''$ and the expression of \eqref{norm-LpNA} defines a norm on $\cal{M}$.
\end{cor}

\begin{proof}
According to Proposition \ref{prop:norm-homo}, it suffices to prove the triangle inequality. We have $\cal{M}=\A+\i \A$.  Let $\cal{N} \ov{\mathrm{def}}{=} \A''$ be the enveloping von Neumann algebra of the associated $\JW$-algebra $\A$. Recall that by \cite[Theorem 1.1.5 p.~17]{ARU97} or \cite[Theorem 2.4 p.~155]{Sto68}, we have $\cal{N}=\cal{R}(\A)+\i \cal{R}(\A)$, where $\cal{R}(\A)$ is the weak* closed real $*$-algebra generated by $\A$. Hence $\cal{M}$ is a subspace of $\cal{N}$. By \cite[Corollary 1.2.10 p.~35]{ARU97}, the trace $\tau$ extends to a normal finite faithful trace on $\cal{N}$, still denoted $\tau$. By Proposition \ref{prop:norm-well-defined}, $\norm{\cdot}_{\L^p(\cal{N})}$ is a norm on the von Neumann algebra $\cal{N}$. Since its restriction on $\cal{M}$ is obviously $\norm{\cdot}_{\L^p(\cal{M})}$, we conclude that $\norm{\cdot}_{\L^p(\cal{M})}$ is a norm on $\cal{M}$.
\end{proof}

\begin{remark} \normalfont
According to \cite[Corollary 6.5 p.~180]{Sto66} and \cite[Theorem 6.6 p.~181]{Sto66}, a $\JW$-factor is either reversible or of type $\I_2$. Recall that by \cite[Corollary 4.30 p.~117]{AlS03}, a $\JW$-algebra $\A$ is reversible if and only if the $\I_2$ summand of $\A$ is reversible. We refer to \cite[Theorem 6.2.5 p.~141]{HOS84} for information on the reversibility of real spin factors.
\end{remark}

Following \cite[Definition 3.21 p.~86]{AlS03}, we say that a $\JBW$-algebra $\A$ is of type $\I_n$, where $n$ is some cardinal number, if there exists an orthogonal family $(e_i)_{i \in I}$ of abelian projections in $\A$ with $\card I=n$ such that $1=\sum_{i \in I} e_i$ and central cover $c(e_i)=1$ for any $i \in I$. We refer to \cite[Theorem 3.39 p.~95]{AlS03} for the classification of $\JBW$-factors of type $\I_n$.

\begin{example} \normalfont
\label{ex-type-II-JW}
The $\JBW$-algebras of type $\I_2$ were classified by Stacey in \cite[Theorem 2 p.~125]{Sta82}. A $\JBW$-algebra $\A$ is of type $\I_2$ if and only if there exist an index set $I$, a family $(\Omega_i)_{i \in I}$ of (localizable) measure spaces and a family $(S_i)_{i \in I}$ of real spin factors, each of dimension strictly greater than 1 giving an isomorphism 
\begin{equation}
\label{isomorphism-type-I2}
\A
=\oplus_{i \in I} \L^\infty_\R(\Omega_i,S_i).
\end{equation}
It is worth noting that according to \cite[Theorem 6.1.8 p.~138]{HOS84}, if $\A$ is a $\JBW$-algebra then $\A$ is a $\JBW$ factor of type $\I_2$ if and only if $\A$ is isomorphic to a spin factor.
\end{example}

If $H$ is a complex Hilbert space and if $n \geq 1$, we consider the $n$-fold antisymmetric product $\Lambda^n(H)$ of $H$, equipped with the canonical inner product given by \cite[V.33]{Bou03} \cite[Exercise 12.4.39 p.~933]{KaR97b}
$$
\bigl\la \xi_1\wedge\cdots \wedge \xi_n,
\eta_1\wedge \cdots \wedge \eta_n\bigr \ra 
\ov{\mathrm{def}}{=}  \det\bigl[\langle \xi_i,\eta_j\ra\bigr],\qquad \xi_1,\ldots,\xi_n, \eta_1,\ldots,\eta_n \in H.
$$
By convention, we put $\Lambda^0(H) \ov{\mathrm{def}}{=} \mathbb{C}$. Let $\Omega$ be the unit element of $\Lambda^0(H)$. Following \cite[Exercise 12.4.40 p.~934]{KaR97b}, we introduce the antisymmetric Fock space over $H$ which is the Hilbertian direct sum $
\Lambda(H)
\ov{\mathrm{def}}{=} {\mathop{\oplus} \limits_{n \geq 0}} \Lambda^n(H)$.  For any $\xi \in H$, the creation operator $c(\xi) \co \Lambda(H) \to \Lambda(H)$ is defined by $c(\xi)\Omega \ov{\mathrm{def}}{=} \xi$ and
$$
c(\xi)(\xi_1\wedge\cdots \wedge \xi_n)
\ov{\mathrm{def}}{=} \xi \wedge \xi_1\wedge\cdots \wedge \xi_n,\qquad \xi_1,\ldots, \xi_n \in H,
$$
and then extending by linearity and continuity. Its adjoint $a(\xi) \ov{\mathrm{def}}{=} c(\xi)^*$ is the <<annihilation operator>>. By \cite[Exercise 12.4.40 p.~934]{KaR97b}, these operators satisfy the canonical anti-commutation relations \cite[p.~217]{EvK98}
$$
c(\xi)a(\eta)+a(\eta)c(\xi)
=\la \xi,\eta\ra_{\cal{H}} \Id_{\Lambda(H)}, 
\quad \text{and} \quad
a(\xi)a(\eta)+a(\eta)a(\xi)
=0,\quad \xi,\eta \in H.
$$
Now, we consider a real Hilbert space $\cal{H}$ and we use the previous construction with $H = \cal{H}_\mathbb{C}$. We let $s(\xi) \ov{\mathrm{def}}{=} a(\xi) + c(\xi)$ for any $\xi \in \cal{H}$. These operators satisfy the equality
$$
s(\xi)s(\eta)+s(\eta)s(\xi)
=2\la \xi,\eta \ra\Id_{\Lambda(H)}, \quad \xi, \eta \in \cal{H}.
$$
The von Neumann Clifford algebra associated with the real Hilbert space $\cal{H}$ is defined by
$$
\Gamma_{-1}(\cal{H}) 
\ov{\mathrm{def}}{=}  \{ s(\xi) : \xi \in \cal{H} \}''.
$$
We equip it with the normal finite faithful trace $\tau$ defined by $\tau(x) \ov{\mathrm{def}}{=} \la x\Omega,\Omega\ra_{\Lambda(H)}$, where $x \in \Gamma_{-1}(\cal{H})$. By \cite[Corollary 2.1 p.~8630]{EfP03}, we have
\begin{equation}
\label{trace-fermion}
\tau(1)=1
\quad \text{and} \quad 
\tau(s(h))=0, \quad h \in \cal{H}.
\end{equation}

For the proof of the next result, we will use the following observation. If $(e_i)_{ i \in I}$ is an orthonormal basis of the Hilbert space $\cal{H}$ then $(s(e_i))_{ i \in I}$ is a family of anticommuting Hermitian unitaries $\not= \pm \Id$ which generates the von Neumann algebra $\Gamma_{-1}(\cal{H})$. In particular, if $E$ is the weak closure of the real span of the family $(s(e_i))_{i \in I}$ then the space $E \oplus \R\Id_{\Lambda(H)}$ equipped with the Jordan product is a $\JW$-algebra isomorphic as a $\JBW$-algebra to a real spin factor as in Example \ref{spin-factor}, see \cite[p.~15]{ARU97}. More precisely, the map $E \oplus \R\Id_{\Lambda(H)} \to \cal{H} \oplus \R 1$, $s(e_k) \mapsto e_k$, $\Id_{\Lambda(H)} \mapsto 1$ defines an isomorphism.  It is essentially observed in \cite[Theorem 2 p.~1061]{Top66}.

\begin{prop}
\label{prop-norm-type-I2}
Let $\cal{M}$ be a $\JBW^*$-algebra such that the associated $\JBW$-algebra $\A$ is of type $\I_2$, equipped with a normal finite faithful trace $\tau$. Suppose that $1 \leq p < \infty$. Then $\cal{M}$ is isomorphic to a $\JW$-subalgebra of a von Neumann algebra $\cal{N}$ such that the trace $\tau$ extends to a normal finite faithful trace on $\cal{N}$. Finally, the expression of \eqref{norm-LpNA} defines a norm on $\cal{M}$.
\end{prop}

\begin{proof}
By Proposition \ref{prop:norm-homo}, we only need to check the triangle inequality. By Example \ref{ex-type-II-JW}, there exist an index set $I$, a family $(\Omega_i)_{i \in I}$ of (localizable) measure spaces and a family $(S_i)_{i \in I}$ of real spin factors, each of dimension strictly greater than 1 giving an isomorphism \eqref{isomorphism-type-I2}. Complexifying the decomposition \eqref{isomorphism-type-I2}, we obtain
$$
\cal{M}
=\oplus_{i \in I} \L^\infty(\Omega_i,\cal{S}_i),
$$ 
where $\cal{S}_{i}$ is the $\JW^*$-algebra (<<complex spin factor>>) associated with the real spin factor $S_i$. We denote by $\cal{H}_i$ the associated real Hilbert space of dimension strictly greater than 1. For any $i \in I$, the trace $\tau$ induces a finite measure $\nu_i$ on $\Omega_i$ and a finite trace $\int_{\Omega_i} \cdot \d \nu_i \ot \tau_i$ on each $\JW$-algebra $\L^\infty_\R(\Omega_i,S_i)$, where $\tau_i$ is the canonical normalized trace on $S_i$ (defined in Example \ref{ex-trace-spin}). 

Now, for any $i \in I$ we introduce the von Neumann algebra $\cal{N}_i \ov{\mathrm{def}}{=} \Gamma_{-1}(\cal{H}_i)$. As previously explained, we may realize each $\cal{S}_i$ as a $\JW^*$-subalgebra of $\cal{N}_i$. Moreover, by \eqref{trace-spin} and \eqref{trace-fermion}, the canonical trace of $\cal{N}_i$  extends the trace $\tau_i$ and consequently we use the same notation $\tau_i$ for this trace.
 
Now, we can also see the $\JW^*$-algebra $\L^\infty(\Omega_i,\cal{S}_i)$ as a $\JW^*$-subalgebra of the $\JW^*$-algebra  associated with the von Neumann algebra $\L^\infty(\Omega_i,\cal{N}_i)$. We equip the von Neumann algebra 
$$
\cal{N}
\ov{\mathrm{def}}{=} \oplus_{i \in I} \L^\infty(\Omega_i,\cal{N}_i)
$$ 
with the trace $\tau_{\cal{N}} \ov{\mathrm{def}}{=} \oplus_{i \in I} (\int_{\Omega_i} \cdot \d \nu_i \ot \tau_i)$. Note that by construction we can see the trace $\tau_{\cal{N}}$ as an extension of the trace $\tau$. By Proposition \ref{prop:norm-well-defined}, $\norm{\cdot}_{\L^p(\cal{N})}$ is a norm on the von Neumann algebra $\cal{N}$. Since its restriction on $\cal{M}$ is obviously $\norm{\cdot}_{\L^p(\cal{M})}$, we conclude that $\norm{\cdot}_{\L^p(\cal{M})}$ is also a norm on $\cal{M}$.
\end{proof}

Now, we combine the previous cases using the classical structure theory of $\JW$-algebras. The proof of the next result  reduces the general situation to a finite direct sum of well-understood summands.

\begin{thm}
\label{thm:JBW-Lp-norm}
Let $\cal{M}$ be a $\JW^*$-algebra equipped with a normal finite faithful trace $\tau$. Suppose that $1 \leq p < \infty$. Then $\cal{M}$ is isomorphic to a $\JW^*$-subalgebra of a von Neumann algebra $\cal{N}$ such that the trace $\tau$ extends to a normal finite faithful trace on $\cal{N}$. Moreover, the expression of \eqref{norm-LpNA} defines a norm on $\cal{M}$.
\end{thm}

\begin{proof}
Consider the $\JW$-algebra $\A$ associated with the $\JW^*$-algebra $\cal{M}$. We have $\cal{M}=\A+\i \A$. We will use general structure results on $\JW$-algebras. By \cite[Theorem 6.4 p.~180]{Sto66} and \cite[Lemma 2.3 p.~154]{Sto68} (see also \cite[Proposition 2.1 p.~1427]{Ayu87}), there exist projections $e_1,e_2$ in the center $\Zc(\A)$ of the $\JW$-algebra $\A$ with sum $e_1 + e_2 = 1$ such that
\begin{enumerate}
	\item $\A_1 \ov{\mathrm{def}}{=} e_1\A$ is reversible,
	\item $\A_2 \ov{\mathrm{def}}{=} e_2\A$ is totally non reversible, hence of type $\I_2$ by \cite[Theorem 6.6 p.~181]{Sto66}.
\end{enumerate}
The trace $\tau$ decomposes as a direct sum $\tau = \tau_1 \oplus \tau_2$ of normal finite faithful traces on the summands. We denote by $\cal{M}_1$ and $\cal{M}_2$ the associated $\JW^*$-algebras. We have an identification $\cal{M}=\cal{M}_1 \oplus \cal{M}_2$. According to Corollary \ref{cor:norm-well-defined-JW}, $\cal{M}_1$ is a $\JW^*$-subalgebra of a von Neumann algebra $\cal{N}_1$ such that the trace $\tau_1$ extends to a normal finite faithful trace on $\cal{N}_1$. By Proposition \ref{prop-norm-type-I2}, we can suppose that $\cal{M}_2$ is a $\JW^*$-subalgebra of a von Neumann algebra $\cal{N}_2$ such that the trace $\tau_2$ extends to a normal finite faithful trace on $\cal{N}_2$. We deduce that $\cal{M}$ is isomorphic to a $\JW^*$-subalgebra of a von Neumann algebra $\cal{N}$ such that the trace $\tau$ extends to a normal finite faithful trace on $\cal{N}$.

The formula defining $\norm{\cdot}_{\L^p(\cal{M})}$ is compatible with this direct sum decomposition, i.e.~if $x = x_1+x_2$ is an element of $\cal{M}$ then we have
\[
\norm{x}_{\L^p(\cal{M})}
\ov{\eqref{norm-LpNA}}{=} \big(\norm{x_1}_{\L^p(\cal{M}_1)}^p+\norm{x_2}_{\L^p(\cal{M}_2)}^p\big)^{\frac{1}{p}}.
\]
Since $\norm{\cdot}_{\L^p(\cal{M}_1)}$ and $\norm{\cdot}_{\L^p(\cal{M}_2)}$ are norms, the $\ell^p$-sum of these norms is again a norm on the direct sum by \cite[Exercise 1.88 p.~68]{Meg98}. This proves the result.
%
%
\end{proof}

\section{Comparison of the norms in the case of $\JW^*$-algebras associated with von Neumann algebras}
\label{sec-VNA}

Let $\cal{M}$ be a von Neumann algebra equipped with a normal finite faithful trace $\tau$. In this section, we investigate the links between the Banach space $\L^{p,A}(\cal{M}) \ov{\mathrm{def}}{=} (\cal{M},\cal{M}_*)_{\frac{1}{p}}$ defined in \eqref{Def-Lp-state-JBW-intro} by interpolation and the space $\L^p(\cal{M})$ defined as the completion of $\cal{M}$ for the norm $
\norm{x}_{\L^p(\cal{M})}
\ov{\mathrm{def}}{=} \big(\tau \big[(x^* \circ x)^{\frac{p}{2}}\big]\big)^{\frac{1}{p}}$. Note that in this context, we have an isometric isomorphism
\begin{equation}
\label{Dixmier-vs-Arhancet}
\L^{p,A}(\cal{M})\ov{\eqref{arh=Dix-if-trace}}{=} \L^{p,K}(\cal{M})
\ov{\eqref{Kosaki-vs-Dixmier} }{=} \L^{p,D}(\cal{M}), 
\end{equation}
where $\L^{p,D}(\cal{M})$ is the Dixmier noncommutative $\L^p$-space, defined in \eqref{Dixmier-NC-spaces}. Moreover, for any $x,y \in \cal{M}$, we have the following elementary observation
\begin{equation}
\label{trace-Jordan}
\tau(x \circ y) 
\ov{\eqref{Jordan-product}}{=} \frac{1}{2} \tau(xy + yx) 
= \frac{1}{2} (\tau(xy) + \tau(xy)) 
= \tau(xy).
\end{equation}

\paragraph{Submajorization} In this section, we will use the language of submajorization. Here, we recall some background. Let $\cal{M}$ be a von Neumann algebra acting on a complex Hilbert space $H$, equipped with a normal semifinite faithful trace $\tau$. It is well-known, e.g.~\cite[Theorem 9.13 p.~494]{Mor17}, that any unbounded selfadjoint operator $x \co \dom(x) \subset H \to H$ admits a unique spectral measure $e^x \co \cal{B}(\R) \to \B(H)$ such that
$$
x 
=\int_\R \lambda \d e^x(\lambda).
$$
A closed densely defined unbounded operator $x \co \dom(x) \to H$ is said to be affiliated with $\cal{M}$ \cite[p.~49]{DoS24} \cite[Definition 5.6.2 p.~542]{KaR97a} if $ux = xu$ for all unitary $u \in \cal{M}'$.  
According to \cite[Proposition 2.1.4 (vi) p.~50]{DoS24} and \cite[Theorem 6.11 p.~404]{KaR97a}, a closed densely defined operator $x$, with polar decomposition $x = v |x|$, is affiliated with $\cal{M}$ if and only if $v \in \cal{M}$ and $|x|$ is affiliated with $\cal{M}$. By \cite[Proposition 2.1.4 (iv) (v) p.~50]{DoS24}, a selfadjoint operator $x$ is affiliated with $\cal{M}$ if and only if $e^x(B) \in \cal{M}$ for all Borel subset $B$ of $\mathbb{C}$ and in this case the operator $f(x)$ is affiliated with $\cal{M}$ for any Borel function $f \co \mathbb{C} \to \mathbb{C}$.

\begin{example} \normalfont
Consider the case $\cal{M} = \B(H)$. Since $\cal{M}' = \mathbb{C}1$, it is obvious that all closed densely defined operators acting on $H$ are affiliated with $\cal{M}$.
\end{example}

A closed densely defined unbounded operator $x \co \dom(x) \to H$ affiliated with $\cal{M}$ is said to be $\tau$-measurable \cite[Proposition 2.3.6 p.~72]{DoS24} if there exists $\lambda > 0$ such that $\tau(e^{|x|}(\lambda,\infty)) < \infty$. The set of all $\tau$-measurable operators is denoted by $S(\cal{M},\tau)$ and is a unital $*$-algebra by \cite[Theorem 2.3.8 p.~73]{DoS24} with respect to strong sums and products, denoted simply by $x+y$ and $xy$ for all $x,y \in S(\cal{M},\tau)$.

\begin{example} \normalfont
If $\tau$ is finite, all closed densely defined unbounded operators affiliated with $\cal{M}$ are $\tau$-measurable.
\end{example}

\begin{example} \normalfont
Consider the case $\cal{M} = \B(H)$ equipped with its standard trace $\tr \co \B(H)_+ \to [0,\infty]$. By \cite[Example 2.3.13 p.~75]{DoS24}, we have $S(\B(H),\tr)=\B(H)$.  
\end{example}

Following \cite[Definition 3.2.1 p.~129]{DoS24}, we define the generalized singular value function $\mu(x) \co [0,\infty] \to [0,\infty]$, $t \mapsto \mu_t(x)$ of a $\tau$-measurable operator $x$ by
$$
\mu_t(x)
\ov{\mathrm{def}}{=} \inf \big\{\lambda \geq 0 : \tau\big(e^{|x|}(\lambda,\infty)\big) \leq t \big\}, \quad t \geq 0.
$$
This function is decreasing and right-continuous. By \cite[p.~146]{DoS24}, the trace $\tau$ admits an extension $\tau \co S(\cal{M},\tau)_+  \to [0,\infty]$ on the set $S(\cal{M},\tau)_+$ of positive $\tau$-measurable operators defined by
\begin{equation}
\label{trace-mut}
\tau(x)
=\int_{0}^{\infty} \mu_t(x) \d t, \quad x \in S(\cal{M},\tau)_+.
\end{equation}
For any $\tau$-measurable operators $x$ and $y$, following \cite[Definition 3.9.1 p.~205]{DoS24} \cite{DeP07} we say that $x$ is submajorized by $y$, denoted by $x \prec \prec y$, if
\begin{equation}
\label{submajorized}
\int_0^s \mu_t(x) \d t 
\leq \int_0^s \mu_t(y) \d t, \quad s \geq 0.
\end{equation}
Let $f \co [0, \infty) \to [0, \infty)$ be a concave function. For any \textit{positive} $\tau$-measurable operators $x$ and $y$, we have by \cite[Theorem 5.2 p.~120]{DoS09}
\begin{equation}
\label{Dodds-Sukochev-concave}
f(x + y) 
\prec \prec f(x) + f(y).
\end{equation}
Let $f \co [0, \infty) \to [0, \infty)$ be an increasing convex function with $f(0)=0$. For any \textit{positive} $\tau$-measurable operators $x$ and $y$, we have by \cite[Theorem 4.5 p.~119]{DoS09}
\begin{equation}
\label{Dodds-Sukochev-convex}
f(x) + f(y)
\prec \prec f(x + y).
\end{equation}
The following result is a generalization of McCarthy's inequalities \cite[Lemma 2.6 p.~260]{McC67} (see also \cite{Rot67} and \cite[Theorem IV.2.14 p.~98]{Bha97}).

\begin{prop}
\label{prop:rotfeld-vna}
Let $\cal{M}$ be a von Neumann algebra equipped with a normal semifinite faithful trace $\tau$. Let $x$ and $y$ be positive $\tau$-measurable operators.
\begin{enumerate}
\item If $0 < r \leq 1$ and if the elements $x^r$ and $y^r$ belong to $\L^1(\cal{M})$ then
\begin{equation}
\label{Carthy-concave}
\tau\bigl((x+y)^r\bigr)
\leq \tau(x^r)+\tau(y^r).
\end{equation}

\item If $r \geq 1$ and if $(x+y)^r \in \L^1(\cal{M})$ then
\begin{equation}
\label{Carthy-convex}
\tau\bigl((x+y)^r\bigr)
\geq \tau(x^r)+\tau(y^r).
\end{equation}
\end{enumerate}
\end{prop}

\begin{proof}
1. The function $f \co t \mapsto t^r$ is concave on $[0,\infty)$. We obtain the submajorization
\begin{equation}
\label{eq:concave-submaj}
(x+y)^r 
= f(x+y) 
\ov{\eqref{Dodds-Sukochev-concave}}{\prec \prec} f(x)+f(y)
=x^r+y^r.
\end{equation}
With \eqref{submajorized}, we deduce that
$$
\int_0^s \mu_t((x+y)^r) \d t 
\leq \int_0^s \mu_t(x^r+y^r) \d t, \quad s \geq 0.
$$
Applying the trace and passing to the limit yields $(x+y)^r \in \L^1(\cal{M})$ and
\[
\tau\bigl((x+y)^r\bigr)
\ov{\eqref{trace-mut}}{=} \int_{0}^{\infty} \mu_t((x+y)^r) \d t
\leq \int_{0}^{\infty} \mu_t(x^r+y^r) \d t
\ov{\eqref{trace-mut}}{=} \tau(x^r+y^r)
=\tau(x^r)+\tau(y^r).
\]

2. Note that the function $f \co t \mapsto t^r$ is increasing and convex on $[0,\infty)$ and satisfies $f(0)=0$. We infer that
\begin{equation}
\label{eq:convex-submaj}
 x^r+y^r 
=f(x)+f(y) 
\ov{\eqref{Dodds-Sukochev-convex}}{\prec \prec} f(x+y)
=(x+y)^r.
\end{equation}
With \eqref{submajorized}, we deduce that
$$
\int_0^s \mu_t(x^r+y^r) \d t 
\leq \int_0^s \mu_t((x+y)^r) \d t, \quad s \geq 0.
$$
Finally, we obtain
\[
\tau(x^r)+\tau(y^r)
=\tau(x^r+y^r)
\ov{\eqref{trace-mut}}{=} \int_{0}^{\infty} \mu_t(x^r+y^r) \d t
\leq \int_{0}^{\infty} \mu_t((x+y)^r) \d t
\ov{\eqref{trace-mut}}{=} \tau\bigl((x+y)^r\bigr).
\]
\end{proof}

\begin{prop}
\label{prop-ine-von-Neumann-algebras}
Let $\cal{M}$ be a von Neumann algebra equipped with a normal finite faithful trace $\tau$. 

\begin{enumerate}
\item If $1 \leq p \leq 2$, then for every $x \in \cal{M}$ we have
\begin{equation}
\label{equivalence-1}
\norm{x}_{\L^{p,D}(\cal{M})}
\leq \norm{x}_{\L^p(\cal{M})}
\leq 2^{\frac{1}{p}-\frac{1}{2}}  \norm{x}_{\L^{p,D}(\cal{M})}.
\end{equation}

\item If $2 \leq p <\infty$, then for every $x \in \cal{M}$ we have
\begin{equation}
\label{equivalence-2}
\norm{x}_{\L^p(\cal{M})}
\leq \norm{x}_{\L^{p,D}(\cal{M})}
\leq 2^{\frac12-\frac1p} \norm{x}_{\L^p(\cal{M})}.
\end{equation}

\item For $p = 2$, the two norms coincide for all $x \in \cal{M}$, that is
\[
\norm{x}_{\L^2(\cal{M})}
=\norm{x}_{\L^{2,D}(\cal{M})},
\qquad x \in \cal{M}.
\]

\item Suppose that $1 \leq p < \infty$. For any normal element $x \in \cal{M}$, we have
\[
\norm{x}_{\L^p(\cal{M})}
=\norm{x}_{\L^{p,D}(\cal{M})}.
\]
\item Suppose that $1 \leq p < \infty$. If $\cal{M}$ is commutative then $\L^p(\cal{M})$ is canonically isometrically isomorphic with the Dixmier noncommutative $\L^p$-space $\L^{p,D}(\cal{M})$. 
\end{enumerate}
\end{prop}

\begin{proof}
Let $x \in \cal{M}$. By \cite[Theorem 6.11 p.~404]{KaR97a}
, we can consider its polar decomposition $x = u|x|$, where $u$ is a partial isometry of $\cal{M}$. Then we have
\begin{equation}
\label{cons-polar}
x^*x = |x|^2
\quad \text{and} \quad
xx^* = u |x|^2 u^*.
\end{equation}
Consider the positive operators
\begin{equation}
\label{inter-TRE56}
A \ov{\mathrm{def}}{=} 
x^*x
=|x|^2 
\qquad\text{and}\qquad
B \ov{\mathrm{def}}{=} 
xx^*
=u|x|^2u^*.
\end{equation}
Note that $u^*u$ is the projection on the support of the positive operator $|x|$ (see \cite[Theorem 6.11 p.~404]{KaR97a} and \cite[p.~823]{Pal01}). Moreover, we have 
\begin{equation}
\label{inter-789J-bis}
x^*\circ x 
\ov{\eqref{Jordan-product}}{=} \frac{x^*x+xx^*}{2} 
= \frac{|x|^2 + u |x|^2 u^*}{2}
\ov{\eqref{inter-TRE56}}{=} \frac{A+B}{2}.
\end{equation}

1. Suppose that $0 < p \leq 2$ and consider the real number $r \ov{\mathrm{def}}{=} \frac{p}{2}$ of $(0,1]$. 
Note that the function $t \mapsto t^r$ is operator concave on $[0,\infty)$  (which is well-known, e.g., \cite[p.~42]{KLW24}). Hence
\begin{equation}
\label{inter-concave-p}
\Big(\frac{A+B}{2}\Big)^{r}
\geq \frac{A^{r} + B^{r}}{2}.
\end{equation}
Applying $\tau$ and using the tracial property we get
\begin{align}
\MoveEqLeft
\label{inter-2345676}
\norm{x}_{\L^p(\cal{M})}^p
\ov{\eqref{norm-LpNA}}{=} \tau\big[(x^* \circ x)^{\frac{p}{2}}\big]
\ov{\eqref{inter-789J-bis}}{=} \tau\Big[\Big(\frac{A+B}{2}\Big)^{r}\Big] 
\ov{\eqref{inter-concave-p}}{\geq} \tau\Big[\frac{A^{r} + B^{r}}{2}\Big] \\
&= \frac{\tau(A^{r}) + \tau(B^{r})}{2}
\ov{\eqref{inter-TRE56}}{=} \frac{\tau(|x|^{p}) + \tau\big((u|x|^2u^*)^{\frac{p}{2}}\big)}{2}. \nonumber
\end{align}
Now, since $u^*u$ is the projection on the support of the positive operator $|x|$ we have the equality $(u|x|^2u^*)^{\frac{p}{2}} = u|x|^{p}u^*$ by functional calculus and
\begin{equation}
\label{inter-789078}
\tau\big((u|x|^2u^*)^{\frac{p}{2}}\big)
= \tau(u|x|^{p}u^*)
=\tau(|x|^{p}u^*u)
= \tau(|x|^{p}).
\end{equation}
Consequently, we infer that
\[
\norm{x}_{\L^p(\cal{M})}^p
\ov{\eqref{inter-2345676} \eqref{inter-789078}}{\geq} \frac{\tau(|x|^p) + \tau(|x|^p)}{2}
= \tau(|x|^p)
\ov{\eqref{Dixmier-NC-spaces}}{=} \norm{x}_{\L^{p,D}(\cal{M})}^p.
\]
Therefore, for any $x \in \cal{M}$ we have
\[
\norm{x}_{\L^{p,D}(\cal{M})}
\leq \norm{x}_{\L^p(\cal{M})}.
\]
We prove the second inequality. We have
\begin{equation}
\label{eq:def-nonassoc-Lp}
\norm{x}_{\L^p(\cal{M})}^p
\ov{\eqref{norm-LpNA}}{=} 
\tau\big((x^*\circ x)^{\frac{p}{2}}\big)
\ov{\eqref{inter-789J-bis}}{=}
\tau\Big(\Big(\frac{A+B}{2}\Big)^r\Big)
=
2^{-r} \tau\big((A+B)^r\big).
\end{equation}
Applying \eqref{Carthy-concave} yields
\begin{equation}
\label{inter-tyui8}
\tau\big((A+B)^r\big)
\leq \tau(A^r)+\tau(B^r).
\end{equation}
Since $B \ov{\eqref{inter-TRE56}}{=} uAu^*$, we have $B^r=uA^ru^*$. We deduce that
\begin{equation}
\label{inter-tyui9}
\tau(B^r)
=\tau(uA^ru^*)
=\tau(A^ru^*u)
= \tau(A^r).
\end{equation}
Consequently, we obtain
\begin{equation}
\label{eq:trace-bound}
\tau\big((A+B)^r\big)
\ov{\eqref{inter-tyui8}\eqref{inter-tyui9}}{\leq } 2\tau(A^r).
\end{equation}
Combining \eqref{eq:def-nonassoc-Lp} and \eqref{eq:trace-bound} gives
\[
\norm{x}_{\L^p(\cal{M})}^p
\ov{\eqref{eq:def-nonassoc-Lp}\eqref{eq:trace-bound}}{\leq}
2^{-r}\cdot 2 \tau(A^r)
\ov{\eqref{inter-TRE56}}{=}
2^{1-r}\tau\bigl(|x|^{2r}\bigr)
=2^{1-\frac{p}{2}} \tau(|x|^p)
\ov{\eqref{Dixmier-NC-spaces}}{=} 2^{1-\frac{p}{2}} \norm{x}_{\L^{p,D}(\cal{M})}^p.
\]
Taking the $p$th root yields the second inequality of \eqref{equivalence-1}.

2. If we let $r \ov{\mathrm{def}}{=} \frac{p}{2} \geq 1$, we have
\begin{equation}
\label{inter-91}
\norm{x}_{\L^p(\cal{M})}
\ov{\eqref{norm-LpNA}}{=} \Big(\tau\Big[(x^*\circ x)^{\frac{p}{2}}\Big]\Big)^{\frac{1}{p}}
\ov{\eqref{inter-789J-bis}}{=}\Big(\tau\Big[\Big(\tfrac{A+B}{2}\Big)^r\Big]\Big)^{\frac{1}{p}}
\ov{\eqref{Dixmier-NC-spaces}}{=} \bigg(\norm{\frac{A+B}{2}}_{\L^{r,D}(\cal{M})}^r\bigg)^{\frac{1}{p}}.
\end{equation}
By the triangle inequality in the Banach space $\L^{r,D}(\cal{M})$, we see that
\begin{equation}
\label{inter-ERT}
\norm{\frac{A+B}{2}}_{\L^{r,D}(\cal{M})}
\leq
\frac{\norm{A}_{\L^{r,D}(\cal{M})}+\norm{B}_{\L^{r,D}(\cal{M})}}{2}.
\end{equation}
Raising to the power $r$ and using the scalar convexity of the function $t \mapsto t^r$ on $[0,\infty)$, since $r \geq 1$, we get
\[
\norm{\frac{A+B}{2}}_{\L^{r,D}(\cal{M})}^r
\ov{\eqref{inter-ERT}}{\leq}
\bigg(\frac{\norm{A}_{\L^{r,D}(\cal{M})}+\norm{B}_{\L^{r,D}(\cal{M})}}{2}\bigg)^r
\leq
\frac{\norm{A}_{\L^{r,D}(\cal{M})}^r+\norm{B}_{\L^{r,D}(\cal{M})}^r}{2}.
\]
Since $B \ov{\eqref{inter-TRE56}}{=} u A u^*$, we have $B^r=u A^r u^*$. Using the tracial property of $\tau$, we deduce that 
$$
\norm{B}_{\L^{r,D}(\cal{M})}^r 
\ov{\eqref{Dixmier-NC-spaces}}{=} \tau(B^r)
=\tau(u A^r u^*)
=\tau(A^r u^*u)
=\tau(A^r) 
\ov{\eqref{Dixmier-NC-spaces}}{=} \norm{A}_{\L^{r,D}(\cal{M})}^r.
$$ 
Finally, $A^r \ov{\eqref{inter-TRE56}}{=} (|x|^2)^r=|x|^{2r}=|x|^p$. So
\[
\norm{x}_{\L^p(\cal{M})}
\ov{\eqref{inter-91}}{=} \bigg(\norm{\frac{A+B}{2}}_{\L^{r,D}(\cal{M})}^r\bigg)^{\frac{1}{p}}
\leq \big(\norm{A}_{\L^{r,D}(\cal{M})}^r\big)^{\frac{1}{p}}
=(\tau(|x|^p))^{\frac{1}{p}}
\ov{\eqref{Dixmier-NC-spaces}}{=} \norm{x}_{\L^{p,D}(\cal{M})}.
\]

Now, we prove the second inequality of \eqref{equivalence-2}. We have
\begin{equation}
\label{eq:def-nonassoc-Lp-r}
\norm{x}_{\L^p(\cal{M})}^p
=
\tau\Big(\Big(\frac{A+B}{2}\Big)^r\Big)
=
2^{-r}\tau\big((A+B)^r\big).
\end{equation}
Since $r \geq 1$, we may apply \eqref{Carthy-convex} to the positive operators $A$ and $B$ and obtain
\begin{equation}
\label{eq:convex-AB}
\tau\big((A+B)^r\big)
\ov{\eqref{Carthy-convex}}{\geq}
\tau(A^r)+\tau(B^r).
\end{equation}
As before, $B \ov{\eqref{inter-TRE56}}{=} uAu^*$, hence $B^r = uA^ru^*$. Therefore, we see that
\begin{equation}
\label{inter-234}
\tau(B^r)
=
\tau(uA^ru^*)
=
\tau(A^ru^*u)
=
\tau(A^r).
\end{equation}
Hence
\[
\tau\big((A+B)^r\big)
\ov{\eqref{eq:convex-AB}\eqref{inter-234}}{\geq}
2\tau(A^r).
\]
Combining this estimate with \eqref{eq:def-nonassoc-Lp-r} gives
\[
\norm{x}_{\L^p(\cal{M})}^p
\geq
2^{-r}\cdot 2 \tau(A^r)
=
2^{1-r}\tau(|x|^{2r})
=
2^{1-\frac{p}{2}}\tau(|x|^p)
\ov{\eqref{Dixmier-NC-spaces}}{=}
2^{1-\frac{p}{2}} \norm{x}_{\L^{p,D}(\cal{M})}^p.
\]
Taking the $p$th root yields
\[
\norm{x}_{\L^{p,D}(\cal{M})}
\leq
2^{\frac12-\frac1p}
\norm{x}_{\L^p(\cal{M})}.
\]

3. If $p=2$ then for any $x \in \cal{M}$ we obtain
$$
\norm{x}_{\L^2(\cal{M})}
\ov{\eqref{norm-LpNA}}{=} \big(\tau(x^* \circ x)\big)^{\frac{1}{2}}
\ov{\eqref{trace-Jordan}}{=} \big(\tau(x^* x)\big)^{\frac{1}{2}}
=\tau\big(|x|^2\big)^{\frac{1}{2}}
\ov{\eqref{Dixmier-NC-spaces}}{=} \norm{x}_{\L^{2,D}(\cal{M})}.
$$

4. If the element $x$ is normal, then $x^*x = xx^*$ and
\begin{equation}
\label{normal-4}
x^* \circ x
\ov{\eqref{Jordan-product}}{=} \frac{x^*x+xx^*}{2}
= x^*x.
\end{equation}
So for any $p \geq 1$, we conclude that
\[
\norm{x}_{\L^p(\cal{M})}
\ov{\eqref{norm-LpNA}}{=} \big(\tau\big[(x^* \circ x)^{\frac{p}{2}}\big]\big)^{\frac{1}{p}}
\ov{\eqref{normal-4}}{=} \big(\tau\big[(x^*x)^{\frac{p}{2}}\big]\big)^{\frac{1}{p}}
= \big(\tau(|x|^p)\big)^{\frac{1}{p}}
\ov{\eqref{Dixmier-NC-spaces}}{=} \norm{x}_{\L^{p,D}(\cal{M})}.
\]

5. The last point is a direct consequence of the fourth point.
\end{proof}

Now, we will show that the constant $2^{\frac{1}{p}-\frac{1}{2}}$ in Proposition~\ref{prop-ine-von-Neumann-algebras} is optimal. We will use the two following results. The second result is folklore and we include the argument for the sake of completeness.

\begin{prop}
\label{prop-homo-bis}
Let $\cal{M}_1$ and $\cal{M}_2$ be $\JW^*$-algebras equipped with some normal finite faithful traces $\tau_{\cal{M}_1}$ and $\tau_{\cal{M}_2}$. Consider a normal $*$-Jordan homomorphism $\pi \co \cal{M}_1 \to \cal{M}_2$ which preserves the traces. Suppose that $1 \leq p < \infty$. Then $\pi$ induces an isometry $\pi_p \co \L^p(\cal{M}_1) \to \L^p(\cal{M}_2)$.
\end{prop}

\begin{proof}
Recall that the closed subalgebra of a unital $\JB^*$-algebra generated by a selfadjoint element and 1 is a commutative $\mathrm{C}^*$-algebra, see \cite[Proposition 3.4.1 p.~359]{CGRP14}. Let $x \in \cal{M}_1$. By an argument of functional calculus, we have
\begin{equation}
\label{inter-456}
(\pi(x)^*\circ\pi(x))^{\frac p2}
=(\pi(x^*\circ x))^{\frac p2}
=\pi\big((x^*\circ x)^{\frac p2}\big).
\end{equation}
Hence, we obtain
\begin{align*}
\MoveEqLeft
\norm{\pi(x)}_{\L^p(\cal{M}_2)}^p
\ov{\eqref{norm-LpNA}}{=}
\tau_{\cal{M}_2}\big((\pi(x)^*\circ\pi(x))^{\frac p2}\big) 
\ov{\eqref{inter-456}}{=} \tau_{\cal{M}_2}\big(\pi\big((x^*\circ x)^{\frac p2}\big)\big) \\
&=\tau_{\cal{M}_1}\big((x^*\circ x)^{\frac p2}\big) 
\ov{\eqref{norm-LpNA}}{=}\norm{x}_{\L^p(\cal{M}_1)}^p.
\end{align*}
Thus $\pi$ induces an isometry $\pi_p \co \L^p(\cal{M}_1) \to \L^p(\cal{M}_2)$.
\end{proof}

\begin{prop}
\label{prop-hom}
Let $\cal{M}_1$ and $\cal{M}_2$ be von Neumann algebras equipped with some normal finite faithful traces $\tau_{\cal{M}_1}$ and $\tau_{\cal{M}_2}$. Consider a normal $*$-homomorphism $\pi \co \cal{M}_1 \to \cal{M}_2$ which preserves the traces. Suppose that $1 \leq p < \infty$. Then $\pi$ induces an isometry $\pi_p \co \L^{p,D}(\cal{M}_1) \to \L^{p,D}(\cal{M}_2)$.
\end{prop}

\begin{proof}
Let $x \in \cal{M}_1$. By an argument of functional calculus, we have
\begin{equation}
\label{inter-456-bis}
(\pi(x)^*\pi(x))^{\frac p2}
=(\pi(x^*x))^{\frac p2}
=\pi\big((x^* x)^{\frac p2}\big).
\end{equation}
Hence, we obtain
\begin{align*}
\MoveEqLeft
\norm{\pi(x)}_{\L^{p,D}(\cal{M}_2)}^p
\ov{\eqref{Dixmier-NC-spaces}}{=}
\tau_{\cal{M}_2}\big((\pi(x)^*\pi(x))^{\frac p2}\big) 
\ov{\eqref{inter-456-bis}}{=} \tau_{\cal{M}_2}\big(\pi\big((x^* x)^{\frac p2}\big)\big) \\
&=\tau_{\cal{M}_1}\big((x^*x)^{\frac p2}\big) 
\ov{\eqref{Dixmier-NC-spaces}}{=}\norm{x}_{\L^{p,D}(\cal{M}_1)}^p.
\end{align*}
We conclude that $\pi$ induces an isometry $\pi_p \co \L^{p,D}(\cal{M}_1) \to \L^{p,D}(\cal{M}_2)$.
\end{proof}

\begin{remark} \normalfont
\label{remark-optimal}
The constant of Proposition~\ref{prop-ine-von-Neumann-algebras} are optimal. Indeed, consider the complex Hilbert space $H = \mathbb{C}^2$ and the element $
x 
\ov{\mathrm{def}}{=} 
\begin{bmatrix}
0 & 1\\
0 & 0
\end{bmatrix}$ of the von Neumann algebra $\B(H)=\M_2(\mathbb{C})$. Then $x^*
=
\begin{bmatrix}
0 & 0\\
1 & 0
\end{bmatrix}$, $
x^*x
=
\begin{bmatrix}
0 & 0\\
0 & 1 
\end{bmatrix}$, $
xx^*
=
\begin{bmatrix}
1 & 0\\
0 & 0
\end{bmatrix}$ 
and $|x|=(x^*x)^{\frac{1}{2}}=\begin{bmatrix}
0 & 0\\
0 & 1 
\end{bmatrix}$. Consequently, we have
\begin{equation}
\label{inter-534}
x^* \circ x
\ov{\eqref{Jordan-product}}{=} \frac{x^*x+xx^*}{2}
=\frac12 \I_2.
\end{equation}
If $1 \leq p < \infty$, we deduce that
\begin{equation}
\label{inter-987}
\norm{x}_{\L^p(\B(H))}^p
\ov{\eqref{norm-LpNA}}{=} 
\Tr\big[(x^* \circ x)^{\frac{p}{2}}\big]
\ov{\eqref{inter-534}}{=}
\Tr\Big[\Big(\tfrac12 \I_2\Big)^{\frac{p}{2}}\Big]
=
2\Big(\frac12\Big)^{\frac{p}{2}}
=
2^{1-\frac{p}{2}}.
\end{equation}
On the other hand, the Schatten $p$-norm is
\[
\norm{x}_{S^p(H)}
=\norm{x}_{\L^{p,D}(\B(H))}
\ov{\eqref{Dixmier-NC-spaces}}{=} (\tr |x|^p)^{\frac{1}{p}}
=1.
\]
Thus, we obtain
$$
\frac{\norm{x}_{\L^p(\B(H))}}{\norm{x}_{S^p(H)}}
=
2^{\frac{1}{p}-\frac12}.
$$
Therefore, if $1 \leq p \leq 2$ the constant $2^{\frac{1}{p}-\frac{1}{2}}$ in the first part of Proposition~\ref{prop-ine-von-Neumann-algebras} cannot be improved. On the other hand, if $2 \leq p < \infty$, the same computation gives
\[
\frac{\norm{x}_{S^p(H)}}{\norm{x}_{\L^p(\B(H))}}
=
2^{\frac12-\frac{1}{p}}.
\]
Therefore, the constant $2^{\frac12-\frac1p}$ in the second part of Proposition~\ref{prop-ine-von-Neumann-algebras} cannot also be improved.

Let $\cal{M}$ be a non-abelian semifinite von Neumann algebra equipped with a normal semifinite faithful trace $\tau$. One can construct an isometric embedding of a scaled copy of $S^p_2$ into the Dixmier noncommutative $\L^p$-space $\L^{p,D}(\cal{M})$ and obtain a similar result. Indeed, by \cite[Lemma 2.3]{LeZ22}, there exists a non-zero $*$-homomorphism $\pi \co \M_2(\mathbb{C}) \to \cal{M}$ whose image is contained in $\cal{M} \cap \L^{1,D}(\cal{M})$. Then $\tau' \ov{\mathrm{def}}{=} \tau \circ \pi \co \M_2(\mathbb{C}) \to \mathbb{C}$ is a non-zero normal trace on the factor $\M_2(\mathbb{C})$. By \cite[Theorem 8.2.8. p.~517]{KaR97b}, we deduce that there exists $\lambda > 0$ such that $\tau'=\lambda \tr$. Note that $\pi$ preserves the trace $\tau$ and $\tau'$ by construction so it is injective by \cite[Lemma 2.1 p.~2283]{Arh19}. By Proposition \ref{prop-hom}, the map $\lambda^{-\frac{1}{p}} \pi \co S^p_2 \to \L^{p,D}(\cal{M})$ is an isometry. According to Proposition \ref{prop-homo-bis}, the map $\lambda^{-\frac{1}{p}} \pi \co \L^p(\M_2(\mathbb{C})) \to \L^p(\cal{M})$ is also an isometry. We conclude that the constants are also optimal in this more general setting.
%
\end{remark}


\section{Isomorphism between the spectral $\L^p$-spaces and the interpolation $\L^p$-spaces}
\label{sec-isomorphism}

The goal of this section is to compare the spectral nonassociative $\L^p$-space $\L^{p}(\cal{M})$ with the interpolation nonassociative $\L^p$-space $\L^{p,A}(\cal{M})$ constructed from the pair $(\cal{M},\cal{M}_*)$, in the case where $\cal{M}$ is a $\JW^*$-algebra. The strategy is to embed the given $\JW^*$-algebra into a suitable finite von Neumann algebra $\cal{N}$. Then, we exploit the existence of a trace-preserving Jordan conditional expectation, which allows us to realize $\L^{p,A}(\cal{M})$ as a positively 1-complemented subspace of the Dixmier noncommutative $\L^p$-space $\L^{p,D}(\cal{N})$ and $\L^{p}(\cal{M})$ as a subspace of the space $\L^{p}(\cal{N})$. We therefore briefly recall the relevant facts about this notion. Finally, we transfer the comparison estimates from Section \ref{sec-VNA}.

\paragraph{Jordan conditional expectations} We say that a positive map $T \co \cal{N} \to \cal{N}$ on a $\JB^*$-algebra $\cal{N}$ is faithful if $T(x)=0$ for some $x \in \cal{N}_+$ implies $x=0$. 

Let $\cal{N}$ be a unital $\JW^*$-subalgebra of a $\JW^*$-algebra $\cal{M}$. A linear map $Q \co \cal{M} \to \cal{M}$ with range $\cal{N}$ is called a Jordan conditional expectation onto $\cal{N}$ if it is unital, positive, and $\cal{N}$-modular, that is,
\begin{equation}
\label{Def-cond-exp-JBstar}
Q(x \circ Q(y))
=Q(x) \circ Q(y), \quad x,y \in \cal{M}.
\end{equation}
Setting \(x=1\) in \eqref{Def-cond-exp-JBstar}, we obtain
\[
Q(Q(y))=Q(y), \qquad y \in \cal M.
\]
Hence $Q$ is a projection. Since the range of $Q$ is $\cal N$, it follows that $Q$ is the identity on $\cal{N}$.

Let $\cal{N}$ be a $\JBW^*$-algebra equipped with a normal finite faithful trace $\tau$.  Suppose that $1 \leq p < \infty$. Let $Q \co \cal{N} \to \cal{N}$ be a normal contractive projection from $\cal{N}$ onto a $\JBW^*$-subalgebra $\cal{M}$ satisfying $\tau \circ Q = \tau$. Then by \cite[Proposition 3.11]{Arh24a} the linear map $Q$ induces a contractive projection $Q_p \co \L^{p,A}(\cal{N}) \to \L^{p,A}(\cal{N})$ whose range is isometric to the nonassociative space $\L^{p,A}(\cal{M},\tau|_{\cal M})$.

\paragraph{Spectral versus interpolation norms}
Now, we can state the main comparison result of this paper. This result shows that the spectral and interpolation norms are always equivalent on a $\JW^*$-algebra equipped with a normal finite faithful trace. The proof consists of two steps. First, we realize the interpolation $\L^p$-space $\L^{p,A}(\cal{M})$ as a positively complemented subspace of the noncommutative $\L^p$-space $\L^{p,D}(\cal{N})$ for a suitable finite von Neumann algebra $\cal{N}$. Second, we use the comparison estimates already established in Section \ref{sec-VNA}. As a byproduct, the third part answers \cite[Conjecture 5.3]{Arh24a} positively in the finite case.

\begin{thm}
\label{Th-comparison-interpolation}
Let $\cal{M}$ be a $\JW^*$-algebra equipped with a normal finite faithful trace $\tau$.

\begin{enumerate}
\item If $1 \leq p \leq 2$, then for every $x \in \cal{M}$,
\begin{equation}
\label{equivalence-interpolation-1}
\norm{x}_{\L^{p,A}(\cal{M})}
\leq \norm{x}_{\L^p(\cal{M})}
\leq 2^{\frac{1}{p}-\frac{1}{2}} \norm{x}_{\L^{p,A}(\cal{M})}.
\end{equation}

\item If $2 \leq p < \infty$, then for every $x \in \cal{M}$,
\begin{equation}
\label{equivalence-interpolation-2}
\norm{x}_{\L^p(\cal{M})}
\leq \norm{x}_{\L^{p,A}(\cal{M})}
\leq 2^{\frac12-\frac1p} \norm{x}_{\L^p(\cal{M})}.
\end{equation}

\item The Banach space $\L^{p,A}(\cal{M})$ is isometric to a positively $1$-complemented subspace of a Dixmier noncommutative $\L^p$-space $\L^{p,D}(\cal{N})$ associated with a finite von Neumann algebra $\cal{N}$, equipped with a suitable normal finite faithful trace.
\end{enumerate}
\end{thm}

\begin{proof}
By Theorem \ref{thm:JBW-Lp-norm}, there exist a finite von Neumann algebra $\cal{N}$ equipped with a normal finite faithful trace and a trace preserving normal $*$-homomorphism $J \co \cal{M} \to \cal{N}$ from $\cal{M}$ into the $\JW^*$-algebra associated with $\cal{N}$. Replacing $\cal{M}$ with its image $J(\cal{M})$, we may see $\cal{M}$ as a $\JW^*$-subalgebra of the $\JW^*$-algebra associated with $\cal{N}$. By Proposition \ref{prop-homo-bis}, the map $J$ induces an isometric embedding $J \co \L^p(\cal{M}) \to \L^p(\cal{N})$.

By \cite[Proposition 2.5]{Arh24a}, there exist a trace preserving normal faithful Jordan conditional expectation $Q \co \cal{N} \to \cal{N}$ onto the $\JW^*$-subalgebra $\cal{M}$. Then by \cite[Proposition 3.11]{Arh24a} the linear map $Q$ induces a contractive projection $Q_p \co \L^{p,D}(\cal{N}) \ov{\eqref{Dixmier-vs-Arhancet}}{=} \L^{p,A}(\cal{N}) \to \L^{p,A}(\cal{N}) \ov{\eqref{Dixmier-vs-Arhancet}}{=} \L^{p,D}(\cal{N})$ onto a subspace $X_p$ containing $\cal{M}$ and isometric to the nonassociative $\L^p$-space $\L^{p,A}(\cal{M},\tau|_\cal{M})=\L^{p,A}(\cal{M})$ via an isometry $\cal{I}_p \co \L^{p,A}(\cal{M}) \to X_p \subset \L^{p,A}(\cal{N})$ such that $\cal{I}_p|_{\cal{M}}=J$. It is clear that $Q_p$ is positive since $Q$ is positive. This proves the third assertion, since $X_p=\Ran Q_p$ is a positively $1$-complemented subspace of $\L^{p,A}(\cal{N}) \ov{\eqref{Dixmier-vs-Arhancet}}{=} \L^{p,D}(\cal N)$ isometric to $\L^{p,A}(\cal M)$.

We first treat the case $1 \leq p \leq 2$. Using Proposition \ref{prop-ine-von-Neumann-algebras}, for any $x \in \cal{M}$, we deduce the inequalities
\begin{align*}
\MoveEqLeft
\norm{x}_{\L^{p}(\cal{M})}
=\norm{J(x)}_{\L^{p}(\cal{N})} 
\ov{\eqref{equivalence-1}}{\leq}  2^{\frac{1}{p}-\frac{1}{2}}  \norm{J(x)}_{\L^{p,D}(\cal{N})} \\
&\ov{\eqref{Dixmier-vs-Arhancet}}{=} 2^{\frac{1}{p}-\frac{1}{2}} \norm{J(x)}_{\L^{p,A}(\cal{N})}
=2^{\frac{1}{p}-\frac{1}{2}} \norm{\cal{I}_p(x)}_{\L^{p,A}(\cal{N})}
=2^{\frac{1}{p}-\frac{1}{2}} \norm{x}_{\L^{p,A}(\cal{M})}         
\end{align*}
and
\begin{align*}
\MoveEqLeft
\norm{x}_{\L^{p,A}(\cal{M})}
=\norm{\cal{I}_p(x)}_{\L^{p,A}(\cal{N})}
= \norm{J(x)}_{\L^{p,A}(\cal{N})}
\ov{\eqref{Dixmier-vs-Arhancet}}{=} \norm{J(x)}_{\L^{p,D}(\cal{N})} \\
&\ov{\eqref{equivalence-1}}{\leq} \norm{J(x)}_{\L^{p}(\cal{N})}
=\norm{x}_{\L^{p}(\cal{M})}.          
\end{align*}
Now, we then treat the range $2 \leq p < \infty$, using the second part of Proposition~\ref{prop-ine-von-Neumann-algebras}. For any $x \in \cal{M}$, we have the inequalities
\begin{align*}
\MoveEqLeft
\norm{x}_{\L^{p}(\cal{M})}
=\norm{J(x)}_{\L^{p}(\cal{N})} 
\ov{\eqref{equivalence-2}}{\leq}    \norm{J(x)}_{\L^{p,D}(\cal{N})} \\
&\ov{\eqref{Dixmier-vs-Arhancet}}{=}  \norm{J(x)}_{\L^{p,A}(\cal{N})}
= \norm{\cal{I}_p(x)}_{\L^{p,A}(\cal{N})}
= \norm{x}_{\L^{p,A}(\cal{M})}         
\end{align*}
and
\begin{align*}
\MoveEqLeft
\norm{x}_{\L^{p,A}(\cal{M})}
=\norm{\cal{I}_p(x)}_{\L^{p,A}(\cal{N})}
= \norm{J(x)}_{\L^{p,A}(\cal{N})}
\ov{\eqref{Dixmier-vs-Arhancet}}{=} \norm{J(x)}_{\L^{p,D}(\cal{N})} \\
&\ov{\eqref{equivalence-2}}{\leq} 2^{\frac{1}{2}-\frac{1}{p}}\norm{J(x)}_{\L^{p}(\cal{N})}
=2^{\frac{1}{2}-\frac{1}{p}}\norm{x}_{\L^{p}(\cal{M})}.          
\end{align*}
In particular, the identity map on \(\cal M\) extends to an isomorphism between the Banach spaces $\L^p(\cal M)$ and $\L^{p,A}(\cal M)$.
\end{proof}

\section{A remark on the duality of spectral nonassociative $\L^p$-spaces}
\label{sec-duality}

In the associative setting of von Neumann algebras, the canonical pairing $\la x,y\ra_{\cal{M},\L^{1,D}(\cal{M})}= \tau(xy)$ identifies the Banach space $\L^{1,D}(\cal{M})$ isometrically with the predual of $\cal{M}$. It is therefore natural to ask whether the spectral nonassociative $\L^p$-spaces introduced above enjoy a similar duality property with respect to the Jordan pairing
\[
\la x,y \ra 
\ov{\mathrm{def}}{=} \tau(x \circ y).
\]
The following proposition shows that, in general, this duality fails to be isometric.

\begin{prop}
\label{prop-no-isometry-duality}
In general, the previous bracket does not induce an isometric embedding $\L^1(\cal{M}) \hookrightarrow (\L^\infty(\cal{M}))^*$.
\end{prop}

\begin{proof}
Consider the von Neumann algebra $\cal{M}= \M_2(\mathbb{C})$ equipped with its standard trace $\tr$ and the element $x=\begin{bmatrix} 
0 & 1 \\ 
0 & 0 
\end{bmatrix}$ of $\cal{M}$. As vector spaces, we have $\L^1(\cal{M})=\cal{M}$. We introduce the linear functional $\phi_x \co \cal{M} \to \mathbb{C}$ induced by $x$ via the Jordan product:
\[
\phi_x(y) 
\ov{\mathrm{def}}{=} \tr(x \circ y), \quad y \in \cal{M}.
\]
The norm of this functional is given by
\begin{align*}
\MoveEqLeft
\norm{\phi_x}_{(\L^\infty(\cal{M}))^*} 
= \sup \left\{ |\tr(x \circ y)| : y \in \cal{M}, \norm{y}_\cal{M} \leq 1 \right\} \\        
&\ov{\eqref{trace-Jordan}}{=} \sup \left\{ |\tr(xy)| : y \in \M_2(\mathbb{C}), \norm{y}_{\M_2(\mathbb{C})} \leq 1 \right\}
=\norm{x}_{S^1_2} 
= \tr(|x|), 
\end{align*}
where $|x| \ov{\mathrm{def}}{=} (x^*x)^{\frac{1}{2}}$ and where we use the standard Schatten $1$-norm (recall that $\M_2(\mathbb{C})^*=S^1_2$). Since $|x| = (x^*x)^{\frac{1}{2}}= \begin{bmatrix} 
0 & 0 \\ 
0 & 1 \end{bmatrix}$, we have $
\norm{\phi_x}_{(\L^\infty(\cal{M}))^*} 
= \tr\left( \begin{bmatrix} 
0 & 0 \\ 0 & 1 
\end{bmatrix} \right) = 1$. Since $\norm{x}_{\L^1(\cal{M})} \ov{\eqref{inter-987}}{=} \sqrt{2} \neq 1 = \norm{\phi_x}_{(\L^\infty(\cal{M}))^*}$, we conclude that the duality is not isometric.
\end{proof}

\section{Complex spin factors}
\label{sec-Complex-spin-factors}


In this section, we study complex spin factors and compute explicit formulas for the spectral nonassociative $\L^p$-norm on elements of the form $\lambda 1+h$, where $h$ is selfadjoint and orthogonal to $1$. We begin with the two-dimensional case and then pass to arbitrary dimension.

\paragraph{The case $\dim \cal{H}=1$}
We first analyze the two-dimensional case by identifying the complexified spin factor with $\ell^\infty_2(\mathbb{C})$. Consider a real Hilbert space $\cal{H}$ with $\dim \cal{H}=1$, generated by a unit vector $e$. Following the construction of Example \ref{spin-factor}, we have $\A=\cal{H} \oplus \R 1 = \R e \oplus \R1$. We also introduce the $\JBW^*$-algebra $\cal{M}=\A+\i \A$ associated with $\A$ by \cite[Corollary 5.1.41 p.~15]{CGRP18}. Hence $\cal{M} = \mathbb{C} e \oplus\mathbb{C} 1$. We consider the trace $\tau$ on $\A$ defined by $\tau(1) = 1$ and $\tau(e) = 0$. We still denote by $\tau$ the complex-linear extension of the trace on $\A$ to $\cal{M}$.

In the basis $(1,e)$ of $\A$, the Jordan product is given by
\begin{equation}
\label{spin-laws-1-h}
1 \circ 1 \ov{\eqref{product-spin-factor}}{=} 1,
\qquad
1 \circ e \ov{\eqref{product-spin-factor}}{=} e \circ 1 \ov{\eqref{product-spin-factor}}{=} e,
\qquad
e \circ e \ov{\eqref{product-spin-factor}}{=} 1.
\end{equation}
Moreover, both elements $1$ and $e$ are selfadjoint. 

Recall that $\ell^\infty_2(\mathbb{C})$ is a $\JBW^*$-algebra, with associated $\JBW$-algebra $\ell^\infty_2(\R)$, if it is endowed with the product and involution
\begin{equation}
\label{product}
(\alpha_1,\alpha_2) \circ (\beta_1,\beta_2)
\ov{\mathrm{def}}{=} (\alpha_1 \beta_1,\alpha_2 \beta_2), \quad (\alpha_1,\alpha_2)^* \ov{\mathrm{def}}{=} (\ovl{\alpha_1},\ovl{\alpha_2}).
\end{equation}

\begin{prop}
\label{prop-two-dimensional}
Suppose that $1 \leq p < \infty$. For any element $x = a e+b 1$ of $\cal{M}$ with $a,b \in \mathbb{C}$, we have
\begin{equation}
\label{norm-two-dimensional}
\norm{x}_{\L^p(\cal{M})}
= \bigg(\frac{|a+b|^p + |b-a|^p}{2}\bigg)^{\frac{1}{p}}.
\end{equation}
\end{prop}

\begin{proof} 
Consider the linear map $\pi \co \cal{M} \to \ell^\infty_2(\mathbb{C})$ defined by
\[
\pi(1) \ov{\mathrm{def}}{=} (1,1)
\quad \text{and} \quad
\pi(e) \ov{\mathrm{def}}{=} (1,-1).
\]
For any element $x = a e + b 1$ we have
\begin{equation}
\label{interAZT}
\pi(x) 
= (a+b,b-a).
\end{equation}
It is easy to check that $\pi$ is an isomorphism of $\JBW^*$-algebras. Its inverse is given by
\begin{equation}
\label{inter-Phi-1}
\pi^{-1}(\alpha_1,\alpha_2)
= \frac{\alpha_1+\alpha_2}{2} 1 + \frac{\alpha_1-\alpha_2}{2} e.
\end{equation}
Consider the trace $\tau'$ on the $\JBW^*$-algebra $\ell^\infty_2(\mathbb{C})$ defined by $
\tau'(\alpha_1,\alpha_2)
\ov{\mathrm{def}}{=}\frac{\alpha_1+\alpha_2}{2}$. Note that
\begin{equation}
\label{inter-56}
\tau\big(\pi^{-1}(\alpha_1,\alpha_2)\big)
\ov{\eqref{inter-Phi-1}}{=} \tau\bigg(\frac{\alpha_1+\alpha_2}{2} 1 + \frac{\alpha_1-\alpha_2}{2} e \bigg)
= \frac{\alpha_1+\alpha_2}{2},
\qquad \alpha_1,\alpha_2 \in \mathbb{C}.
\end{equation}
Hence the map $\pi$ is trace preserving. By Proposition \ref{prop-homo-bis}, we deduce that $\pi \co \L^p(\cal{M}) \to \L^p(\ell^\infty_2(\mathbb{C}))$ is an isometry. Consequently, for any element $x = a e + b 1$ of the $\JBW^*$-algebra $\cal{M}$, we have
\begin{align*}
\MoveEqLeft
\norm{x}_{\L^p(\cal{M})}
=\norm{\pi(x)}_{\L^p(\ell^\infty_2(\mathbb{C}))}
=\norm{(a+b,b-a)}_{\L^p(\ell^\infty_2(\mathbb{C}))} \\
&\ov{\eqref{norm-LpNA}}{=} \big(\tau' \big[((a+b,b-a)^* \circ (a+b,b-a))^{\frac{p}{2}}\big]\big)^{\frac{1}{p}} \\
&\ov{\eqref{product}}{=} \big(\tau' \big[(|a+b|^2,|b-a|^2)^{\frac{p}{2}}\big]\big)^{\frac{1}{p}}
=\bigg(\frac{|a+b|^p + |b-a|^p}{2}\bigg)^{\frac{1}{p}}.
\end{align*} 
\end{proof}

Now, we turn to the higher-dimensional case. We will reduce the higher-dimensional computation to the previous two-dimensional situation. 

\paragraph{Arbitrary dimension}
First, we describe a concrete model for the $\JBW^*$-algebra associated with a real spin factor. The abstract complexification result of \cite[Corollary 5.1.41 p.~15]{CGRP18} only yields an existence statement, whereas here we need an explicit model. Let $\cal{V}$ be a complex Hilbert space together with a conjugation $x \mapsto \ovl{x}$ (i.e.~an involutive conjugate-linear map which satisfies $\la \ovl{x},\ovl{y} \ra=\la y,x \ra$ for any $x,y \in \cal{V}$, see \cite[p.~235]{Wei80}), and a distinguished norm-one vector $1 \in \cal{V}$ such that $\ovl{1} = 1$. We equip it with the product 
\begin{equation}
\label{product-complex-spin-factor}
x \circ y
\ov{\mathrm{def}}{=} \la x, 1 \ra y + \la y, 1 \ra x - \la x, \ovl{y} \ra 1, \quad x,y \in \cal{V},
\end{equation}
and the involution
\begin{equation}
\label{involution-complex-spin-factor}
x^* 
\ov{\mathrm{def}}{=} 2 \la 1, x \ra 1 - \ovl{x}, \quad x \in \cal{V}.
\end{equation}
For vectors orthogonal to the unit, the involution and the conjugation are related by the formula $x^*
\ov{\eqref{involution-complex-spin-factor}}{=} -\ovl{x}$, where $x \perp 1$. Here $\perp$ stands for the orthogonality in the Hilbert space $\cal{V}$. Note that
\begin{equation}
\label{adjoint-of-1}
1^*
\ov{\eqref{involution-complex-spin-factor}}{=} 2 \la 1, 1 \ra 1 - \ovl{1}
=1
\end{equation}
and
\begin{equation}
\label{1-unit}
1 \circ y
\ov{\eqref{product-complex-spin-factor}}{=} \la 1, 1 \ra y + \la y, 1 \ra 1 - \la 1, \ovl{y} \ra 1
=y + \la y, 1 \ra 1 - \la \ovl{1}, \ovl{y} \ra 1
= y, \quad y \in \cal{V}.
\end{equation}
It is known that the formula 
\begin{equation}
\label{norm-complex-spin-factor}
\norm{x}^2 
\ov{\mathrm{def}}{=} \la x, x \ra + \left( \la x, x \ra^2 - | \la x, \ovl{x} \ra |^2 \right)^{\frac12}, \quad x \in \cal{V},
\end{equation}
defines a norm $\norm{\cdot}$ on the Hilbert space $\cal{V}$, which is equivalent to the Hilbertian norm $\norm{\cdot}_2 \ov{\mathrm{def}}{=} \la x, x\ra^{\frac{1}{2}}$. Then the vector space $\cal{V}$ equipped with $\norm{x}$, $\circ$ and $*$ is a unital $\JBW^*$-algebra with unit 1, called complex spin factor. We refer to \cite[Section 9.5]{Isi19}, \cite{EPV25} and references therein for more information. The selfadjoint part $\cal{V}_{\sa} \ov{\mathrm{def}}{=} \{x \in \cal{V} : x^* = x \}$, equipped with the product and the norm induced by \eqref{product-complex-spin-factor} and \eqref{norm-complex-spin-factor} is a $\JBW$-algebra. It is known that $\cal{V}_{\sa}$ is isomorphic to a real spin factor by an isomorphism sending $1$ to $1$. Through this identification, we equip $\cal{V}$ with the complex-linear extension of the canonical trace given in \eqref{trace-spin}. From now on, we write $\cal{M}=\cal{V}$.


\begin{prop}
\label{prop:Lp-general-spin-factor}
Suppose that $1 \leq p < \infty$. For any element $x = h+\lambda 1$ of $\cal{M}$, where $\lambda \in \mathbb{C}$, $h$ is selfadjoint, and $h \perp 1$, we have
\begin{equation}
\label{eq:Lp-general-spin}
\norm{x}_{\L^p(\cal{M})}
= \bigg(\frac{|\lambda + \norm{h}_2|^p + |\lambda - \norm{h}_2|^p}{2}\bigg)^{\frac{1}{p}}.
\end{equation}
\end{prop}

\begin{proof}
If $h = 0$, then $x = \lambda 1$ and
\begin{equation}
\label{interUYI}
x^* \circ x
= (\lambda 1)^* \circ (\lambda 1)
= \overline{\lambda} \lambda  (1^* \circ 1)
\ov{\eqref{adjoint-of-1}\eqref{1-unit}}{=}  |\lambda|^2 1.
\end{equation}
In this case, we obtain
\[
\norm{x}_{\L^p(\cal{M})}
\ov{\eqref{norm-LpNA}}{=} \big[\tau\big[(x^* \circ x)^{\frac{p}{2}}\big]\big]^{\frac{1}{p}}
\ov{\eqref{interUYI}}{=} \big[\tau\big(|\lambda|^p 1\big)\big]^{\frac{1}{p}}
= |\lambda|\tau(1)^{\frac{1}{p}}
= |\lambda|.
\]
This proves \eqref{eq:Lp-general-spin} in the case $h=0$. Now, assume that $x = h+\lambda 1$ with $h \neq 0$ and consider the vector
\begin{equation}
\label{def-de-e}
e 
\ov{\mathrm{def}}{=} \frac{h}{\norm{h}_2}.
\end{equation}
We have $e=e^* \ov{\eqref{involution-complex-spin-factor}}{=} 2 \la 1, e \ra 1 - \ovl{e}=-\ovl{e}$. Since $e \perp 1$, we deduce that
$$
e \circ e
\ov{\eqref{product-complex-spin-factor}}{=} \la e, 1 \ra e + \la e, 1 \ra e - \la e, \ovl{e} \ra 1
=1.
$$
The $\JBW^*$-subalgebra $\cal{M}_0$ of the $\JBW^*$-algebra $\cal{M}$ generated by $\{1,e\}$ is isomorphic to the two-dimensional algebra considered in Proposition \ref{prop-two-dimensional}. Note that the element $x$ belongs to $\cal{M}_0$ and can be written as
\[
x 
= \lambda 1 + h 
\ov{\eqref{def-de-e}}{=} \lambda 1 + \norm{h}_2 e.
\]
By Proposition~\ref{prop-two-dimensional}, we know that on $\cal{M}_0$ (equipped with the restriction of the trace) we have
\begin{equation}
\label{inter-46}
\bnorm{\lambda 1 + \norm{h}_2 e}_{\L^p(\cal{M}_0)}
\ov{\eqref{norm-two-dimensional}}{=}  \bigg(\frac{|\lambda + \norm{h}_2|^p + |\lambda - \norm{h}_2|^p}{2}\bigg)^{\frac{1}{p}}.
\end{equation}
Since $x$ belongs to $\cal{M}_0$, the functional calculus involved in the definition of the $\L^p$-norm takes place inside $\cal{M}_0$. Therefore the $\L^p$-norm of $x$ computed in $\cal M$ coincides with its $\L^p$-norm computed in $\cal{M}_0$. Hence
\[
\norm{x}_{\L^p(\cal{M})}
= \norm{x}_{\L^p(\cal{M}_0)}
\ov{\eqref{inter-46}}{=} \bigg(\frac{|\lambda + \norm{h}_2|^p + |\lambda - \norm{h}_2|^p}{2}\bigg)^{\frac{1}{p}}.
\]
\end{proof}

\begin{remark} \normalfont
\label{remark-JBW-triple}
It is a general fact that each $\JBW^*$-algebra $(\cal{M},*,\circ)$ has a canonical structure of $\JBW^*$-triple, in the sense of \cite[Definition 2.5.30 p.~167]{Chu12}, see \cite[Theorem 3.3 p.~283]{BKU78}, \cite[Lemma 3.1.6 p.~174]{Chu12} and \cite[p.~224]{CGRP14}. The triple product is defined by $\{x,y,z\} \ov{\mathrm{def}}{=} (x \circ y^*) \circ z + (y^* \circ z) \circ x - (x \circ z) \circ y^*$. If $\dim \cal{V}=n$ is greater than or equal to 3, the canonical $\JBW^*$-triple associated with $\cal{V}$ is the so-called Cartan factor of type $\IV_n$ and the triple product is given by $\{x,y,z\}=\la x,y\ra z+\la z,y \ra x-\la x, \ovl{z}\ra \ovl{y}$, see \cite[p.~439]{CGRP18}. We refer to the books \cite{Isi19} and \cite[Chapter 3]{Fri05} and to the papers \cite{Har74}, \cite{Har81} and \cite{HeI92} for more information on this $\JBW^*$-triple. Finally, note that if $\dim \cal{V}=1$, $\cal{V}$ identifies to $\mathbb{C}$ and that if $\dim \cal{V}=2$ then $\cal{V}$ identifies to $\ell^\infty_2=\mathbb{C} \oplus \mathbb{C}$, which is not a $\JBW^*$-factor.  
\end{remark}

\begin{example} \normalfont 
\label{ex-sym-2x2}
According to \cite[p.~553]{CGRP18}, the space $\Sym_2(\mathbb{C}) \ov{\mathrm{def}}{=} \bigg\{\begin{bmatrix}
   a  & b  \\
   b  & c  \\
\end{bmatrix} : a,b,c \in \mathbb{C} \bigg\}$ of all symmetric $2 \times 2$ complex matrices can be identified with a complex spin factor of dimension 3 when equipped with the operator norm and the Jordan product $\Sym_2(\mathbb{C}) \times \Sym_2(\mathbb{C}) \to \Sym_2(\mathbb{C})$, $(x,y) \mapsto x \circ y 
\ov{\eqref{Jordan-product}}{=} \frac{1}{2}(xy+yx)$. More precisely, consider the Pauli matrices $\I
\ov{\mathrm{def}}{=}
\begin{bmatrix}
1 & 0\\
0 & 1
\end{bmatrix}$, $
\sigma_1
\ov{\mathrm{def}}{=}
\begin{bmatrix}
1 & 0\\
0 & -1
\end{bmatrix}$, $
\sigma_2
\ov{\mathrm{def}}{=}
\begin{bmatrix}
0 & 1\\
1 & 0
\end{bmatrix}$ as in \cite[p.~140]{HOS84}, which define a basis of the space $\Sym_2(\mathbb{C})$. We have the equalities
\begin{equation}
\label{Pauli-relations}
\sigma_1^*=\sigma_1,
\quad
\sigma_2^*=\sigma_2,
\quad
\sigma_1 \circ \sigma_1 = \I,
\quad
\sigma_2 \circ \sigma_2 = \I
\quad \text{and} \quad
\sigma_1 \circ \sigma_2 = 0.
\end{equation}
Let $\cal{V}$ be the abstract complex spin factor of dimension $3$, namely $
\cal{V}
\ov{\mathrm{def}}{=}
\mathbb{C}1 \oplus \mathbb{C}e_1 \oplus \mathbb{C}e_2$, 
where $\{1,e_1,e_2\}$ is an orthonormal basis for the underlying Hilbert space and the conjugation is given by
\[
\overline{\lambda 1 + z_1 e_1 + z_2 e_2}
\ov{\mathrm{def}}{=}
\overline{\lambda} 1 - \overline{z_1} e_1 - \overline{z_2} e_2.
\]
Then $e_1$ and $e_2$ are selfadjoint, orthogonal to $1$ and satisfy
\[
e_1 \circ e_1 = 1,
\quad
e_2 \circ e_2 = 1,
\quad \text{and} \quad
e_1 \circ e_2 = 0.
\]
Hence the linear map $J \co \cal{V} \to \Sym_2(\mathbb{C})$, $\lambda 1 + z_1 e_1 + z_2 e_2 \mapsto \lambda \I + z_1 \sigma_1 + z_2 \sigma_2$ is an isomorphism of $\JBW^*$-algebras. The associated $\JBW$-algebra is the space $\Sym_2(\R) = \R \I \oplus \R \sigma_1 \oplus \R \sigma_2$ of symmetric $2 \times 2$ real matrices, equipped with the operator norm and the Jordan product. This $\JBW$-algebra is isomorphic to a real spin factor of dimension 3, see \cite[Example 3.36 p.~92]{AlS03}.  

We equip $\Sym_2(\mathbb{C})$ with the normal finite faithful trace $\tau$ defined by 
\[
\tau(x)
\ov{\mathrm{def}}{=}
\frac{1}{2}\tr(x), \quad x \in \Sym_2(\mathbb{C}).
\]
Then we have $\tau(\I)=1$ and $\tau(\sigma_1)=\tau(\sigma_2)=0$. So this trace coincides with the complexified canonical trace of the abstract spin factor introduced in \eqref{trace-spin}.

Consider an element $x=\alpha \I+\beta \sigma_1+\gamma \sigma_2$ of $\Sym_2(\mathbb{C})$, where $\alpha,\beta,\gamma \in \mathbb{C}$. Since $\I$, $\sigma_1$ and $\sigma_2$ are selfadjoint, we have
\[
x^*
=
\overline{\alpha}\I+\overline{\beta}\sigma_1+\overline{\gamma}\sigma_2.
\]
Using the bilinearity of the Jordan product, we obtain
\begin{align*}
\MoveEqLeft
x^* \circ x
=
(\overline{\alpha}\I+\overline{\beta}\sigma_1+\overline{\gamma}\sigma_2)
\circ
(\alpha \I+\beta \sigma_1+\gamma \sigma_2) \\
&=
\overline{\alpha}\alpha (\I \circ \I)
+\overline{\alpha}\beta (\I \circ \sigma_1)
+\overline{\alpha}\gamma (\I \circ \sigma_2) 
+\overline{\beta}\alpha (\sigma_1 \circ \I)
+\overline{\beta}\beta (\sigma_1 \circ \sigma_1)
+\overline{\beta}\gamma (\sigma_1 \circ \sigma_2) \\
&\quad
+\overline{\gamma}\alpha (\sigma_2 \circ \I)
+\overline{\gamma}\beta (\sigma_2 \circ \sigma_1)
+\overline{\gamma}\gamma (\sigma_2 \circ \sigma_2) \\
&\ov{\eqref{Pauli-relations}}{=}
|\alpha|^2 \I
+\overline{\alpha}\beta \sigma_1
+\overline{\alpha}\gamma \sigma_2
+\overline{\beta}\alpha \sigma_1
+|\beta|^2 \I 
+0
+\overline{\gamma}\alpha \sigma_2
+0
+|\gamma|^2 \I \\
&=
\bigl(|\alpha|^2+|\beta|^2+|\gamma|^2\bigr)\I
+
(\overline{\alpha}\beta+\ovl{\beta}\alpha)\sigma_1
+
(\overline{\alpha}\gamma+\ovl{\gamma}\alpha)\sigma_2.
\end{align*}
Finally, since $
\overline{\alpha}\beta+\overline{\beta}\alpha
=
2\Re(\overline{\alpha}\beta)$ and $
\overline{\alpha}\gamma+\overline{\gamma}\alpha
=
2\Re(\overline{\alpha}\gamma)$, 
we conclude that
\begin{equation}
\label{x*x}
x^* \circ x
=
\bigl(|\alpha|^2+|\beta|^2+|\gamma|^2\bigr)\I
+
2\Re(\overline{\alpha}\beta)\sigma_1
+
2\Re(\overline{\alpha}\gamma)\sigma_2.
\end{equation}
Set 
\begin{equation}
\label{def-de-m}
m
\ov{\mathrm{def}}{=}
|\alpha|^2+|\beta|^2+|\gamma|^2, \quad 
u
\ov{\mathrm{def}}{=}
2\Re(\overline{\alpha}\beta)
\quad \text{and} \quad 
v
\ov{\mathrm{def}}{=}
2\Re(\overline{\alpha}\gamma).
\end{equation}
We have $
x^* \circ x
\ov{\eqref{x*x}}{=}
m\I+u\sigma_1+v\sigma_2$. In matrix form, we obtain
\[
x^* \circ x
=
\begin{bmatrix}
m+u & v\\
v & m-u
\end{bmatrix}.
\]
Note that the characteristic polynomial $\chi(t)$ of $x^* \circ x$ is
\[
\chi(t)
=\det
\begin{bmatrix}
m+u-t & v\\
v & m-u-t
\end{bmatrix}
=
(m-t)^2-u^2-v^2.
\]
Therefore the eigenvalues of the selfadjoint matrix $x^* \circ x$ are $
m-r$ and $m+r$, where
\begin{equation}
\label{def-de-r}
r
\ov{\mathrm{def}}{=}
\sqrt{u^2+v^2}
=
2\sqrt{\Re(\overline{\alpha}\beta)^2+\Re(\overline{\alpha}\gamma)^2}.
\end{equation}
Since $x^* \circ x$ is positive (see \cite[p.~9]{CGRP18}), both numbers $m-r$ and $m+r$ are nonnegative. Suppose that $1 \leq p < \infty$. The eigenvalues of the element $(x^* \circ x)^{\frac{p}{2}}$ are $(m-r)^{\frac{p}{2}}$ and $(m+r)^{\frac{p}{2}}$. By definition of the nonassociative $\L^p$-norm, we finally obtain, for any matrix $x \in \Sym_2(\mathbb{C})$, the formula	
\begin{equation}
\label{une-autre-egalite}
\norm{x}_{\L^p(\Sym_2(\mathbb{C}))}
\ov{\eqref{norm-LpNA}}{=}
\big(\tau \big[(x^* \circ x)^{\frac{p}{2}}\big]\big)^{\frac{1}{p}}
=\big(\tfrac{1}{2}\Tr \big[(x^* \circ x)^{\frac{p}{2}}\big]\big)^{\frac{1}{p}}
=\bigg(
\frac{(m-r)^{\frac{p}{2}}+(m+r)^{\frac{p}{2}}}{2}
\bigg)^{\frac{1}{p}}.
\end{equation}
This norm does not coincide in general with the norm induced by the ambient Dixmier noncommutative norm on $\M_2(\mathbb{C})$. Indeed, for any element $x = \alpha \I+\beta \sigma_1+\gamma \sigma_2$, we have
\begin{align}
\MoveEqLeft
\label{egalite-sans-fin-4}
x^*x
=(\alpha \I+\beta \sigma_1+\gamma \sigma_2)^*(\alpha \I+\beta \sigma_1+\gamma \sigma_2) 
=
(\overline{\alpha} \I+\overline{\beta} \sigma_1+\ovl{\gamma} \sigma_2)
(\alpha \I+\beta \sigma_1+\gamma \sigma_2) \\
&=\overline{\alpha}\alpha \I
+\overline{\alpha}\beta \sigma_1
+\overline{\alpha}\gamma \sigma_2
+\overline{\beta}\alpha \sigma_1
+\overline{\beta}\beta \sigma_1^2
+\overline{\beta}\gamma \sigma_1\sigma_2 
+\overline{\gamma}\alpha \sigma_2
+\overline{\gamma}\beta \sigma_2\sigma_1
+\overline{\gamma}\gamma \sigma_2^2 \nonumber \\
&=
|\alpha|^2 \I
+\ovl{\alpha}\beta \sigma_1
+\ovl{\alpha}\gamma \sigma_2
+\ovl{\beta}\alpha \sigma_1
+|\beta|^2 \I
+\ovl{\beta}\gamma (-\i\sigma_3) 
+\ovl{\gamma}\alpha \sigma_2
+\ovl{\gamma}\beta (\i\sigma_3)
+|\gamma|^2 \I \nonumber \\
&=
\bigl(|\alpha|^2+|\beta|^2+|\gamma|^2\bigr)\I
+
(\ovl{\alpha}\beta+\ovl{\beta}\alpha)\sigma_1
+
(\ovl{\alpha}\gamma+\ovl{\gamma}\alpha)\sigma_2 
+
\bigl(-\i\ovl{\beta}\gamma+\i\ovl{\gamma}\beta\bigr)\sigma_3 \nonumber \\
&=
\bigl(|\alpha|^2+|\beta|^2+|\gamma|^2\bigr)\I
+
2\Re(\ovl{\alpha}\beta)\sigma_1
+
2\Re(\ovl{\alpha}\gamma)\sigma_2
+
2\Im(\ovl{\beta}\gamma)\sigma_3, \nonumber
\end{align}
where $
w
\ov{\mathrm{def}}{=} 2\Im(\ovl{\beta}\gamma)$ and $
\sigma_3
\ov{\mathrm{def}}{=} \begin{bmatrix}
0&\i\\
-\i&0
\end{bmatrix}$. 
Therefore, if we use the notations \eqref{def-de-m} and $w
\ov{\mathrm{def}}{=}
2\Im(\ovl{\beta}\gamma)$, 
then
\begin{equation}
\label{eq:x*x-spin}
x^*x
\ov{\eqref{egalite-sans-fin-4}}{=}
m\I+u\sigma_1+v\sigma_2+w\sigma_3
=\begin{bmatrix}
m+u & v+\i w\\
v-\i w & m-u
\end{bmatrix}.
\end{equation}
Hence the characteristic polynomial of the matrix $x^*x$ is given by
\begin{align*}
\chi_{x^*x}(t)
&\ov{\eqref{eq:x*x-spin}}{=}
\det
\begin{bmatrix}
m+u-t & v+\i w\\
v-\i w & m-u-t
\end{bmatrix} 
=(m+u-t)(m-u-t)-(v+\i w)(v-\i w) \\
&=
(m-t)^2-u^2-v^2-w^2.
\end{align*}
If we introduce the positive real number
\begin{equation}
\label{eq:def-s-spin}
s
\ov{\mathrm{def}}{=}
\sqrt{u^2+v^2+w^2}
\ov{\eqref{egalite-sans-fin-4}}{=}
2\sqrt{\Re(\ovl{\alpha}\beta)^2+\Re(\ovl{\alpha}\gamma)^2+\Im(\ovl{\beta}\gamma)^2},
\end{equation}
the eigenvalues of the positive matrix $x^*x$ are $m-s$ and $m+s$. Since $x^*x \geq 0$, both numbers $m-s$ and $m+s$ are nonnegative. It follows that the eigenvalues of $|x|^p=(x^*x)^{\frac p2}$ are $(m-s)^{\frac p2}$ and $(m+s)^{\frac p2}$. Consequently, if $\tau=\frac12 \Tr$ denotes the normalized trace on the matrix algebra $\M_2(\mathbb{C})$, then
\begin{equation}
\label{eq:Dixmier-spin}
\norm{x}_{\L^{p,D}(\M_2(\mathbb{C}),\tau)}
\ov{\eqref{Dixmier-NC-spaces}}{=}
\big(\tau(|x|^p)\big)^{\frac1p}
=
\bigg(
\frac{(m-s)^{\frac p2}+(m+s)^{\frac p2}}{2}
\bigg)^{\frac1p}.
\end{equation}
The nonassociative $\L^p$-norm on the complex spin factor $\Sym_2(\mathbb{C})$ does not coincide in general with the norm induced by the ambient Dixmier noncommutative $\L^p$-space of $\M_2(\mathbb{C})$. Indeed, we will prove in Proposition \ref{prop-spin-comparison} that the two norms coincide exactly when the element $x$ is normal in $\M_2(\mathbb{C})$.
\end{example}

The next result is a consequence of Theorem \ref{Th-comparison-interpolation}. Here, we provide an alternative elementary proof.

\begin{prop}
\label{prop-spin-comparison}
\begin{enumerate}
\item If $1 \leq p \leq 2$, then for every $x \in \Sym_2(\mathbb{C})$,
\[
\norm{x}_{\L^{p,D}(\M_2(\mathbb{C}),\tau)}
\leq
\norm{x}_{\L^p(\Sym_2(\mathbb{C}))}
\leq
2^{\frac1p-\frac12}
\norm{x}_{\L^{p,D}(\M_2(\mathbb{C}),\tau)}.
\]

\item If $2 \leq p < \infty$, then for every $x \in \Sym_2(\mathbb{C})$,
\[
\norm{x}_{\L^p(\Sym_2(\mathbb{C}))}
\leq
\norm{x}_{\L^{p,D}(\M_2(\mathbb{C}),\tau)}
\leq
2^{\frac12-\frac1p}
\norm{x}_{\L^p(\Sym_2(\mathbb{C}))}.
\]
Moreover, the constants $2^{\frac1p-\frac12}$ and $2^{\frac12-\frac1p}$ are optimal.
\end{enumerate}
\end{prop}

\begin{proof}
Consider the function $f_m \co [0,m] \to \R$ defined by
\begin{equation}
\label{def-de-fm}
f_{m}(t)
\ov{\mathrm{def}}{=}
\bigg(
\frac{(m-t)^{\frac p2}+(m+t)^{\frac p2}}{2}
\bigg)^{\frac1p}, \quad  t \in [0,m].
\end{equation}
Then $
\norm{x}_{\L^p(\Sym_2(\mathbb{C}))}
\ov{\eqref{une-autre-egalite}\eqref{def-de-fm}}{=} f_{m}(r)$ and $ 
\norm{x}_{\L^{p,D}(\M_2(\mathbb{C}),\tau)}\ov{\eqref{eq:Dixmier-spin} \eqref{def-de-fm}}{=} f_{m}(s)$. Moreover, we see that
\begin{equation}
\label{quelques-calculs}
f_{m}(0)=m^{\frac12}
\qquad \text{and} \qquad
f_{m}(m)
=2^{\frac12-\frac1p}m^{\frac12}.
\end{equation}
We have
\[
\frac{\d}{\d t}\big[f_{m}(t)^p\big]
\ov{\eqref{def-de-fm}}{=} \frac{\d}{\d t}\bigg[ \frac{(m-t)^{\frac p2}+(m+t)^{\frac p2}}{2}\bigg]
=\frac p4\Big((m+t)^{\frac p2-1}-(m-t)^{\frac p2-1}\Big).
\]
Assume first that $1 \leq p \leq 2$. Then $\frac p2-1 \leq 0$. Hence $(m+t)^{\frac p2-1} \leq (m-t)^{\frac p2-1}$. We deduce that the function $f_{m}$ is decreasing on $[0,m]$. Since $r \leq s$, we get
\[
\norm{x}_{\L^{p,D}(\M_2(\mathbb{C}),\tau)}
\ov{\eqref{eq:Dixmier-spin} \eqref{def-de-fm}}{=} f_{m}(s)
\leq f_{m}(r) \ov{\eqref{une-autre-egalite}\eqref{def-de-fm}}{=} \norm{x}_{\L^p(\Sym_2(\mathbb{C}))}.
\]
Moreover, since $f_{m}$ is decreasing, we have
\begin{align*}
\MoveEqLeft
\norm{x}_{\L^p(\Sym_2(\mathbb{C}))}
=f_{m}(r) 
\leq f_{m}(0)
\ov{\eqref{quelques-calculs}}{=}
m^{\frac{1}{2}} \\
&\ov{\eqref{quelques-calculs}}{=} 2^{\frac1p-\frac12}f_{m}(m)
\leq 2^{\frac1p-\frac12}f_{m}(s)
=2^{\frac1p-\frac12}
\norm{x}_{\L^{p,D}(\M_2(\mathbb{C}),\tau)}.         
\end{align*}
Now, assume that $2 \leq p < \infty$. Then $\frac p2-1 \geq 0$. Hence $
(m+t)^{\frac p2-1} \geq (m-t)^{\frac p2-1}$. We infer that $f_{m}$ is increasing on $[0,m]$. Since $r \leq s$, we obtain
\[
\norm{x}_{\L^p(\Sym_2(\mathbb{C}))} 
\ov{\eqref{une-autre-egalite}\eqref{def-de-fm}}{=} f_{m}(r)
\leq f_{m}(s) \ov{\eqref{eq:Dixmier-spin} \eqref{def-de-fm}}{=} \norm{x}_{\L^{p,D}(\M_2(\mathbb{C}),\tau)}.
\]
Furthermore, we have
\begin{align*}
\MoveEqLeft
\norm{x}_{\L^{p,D}(\M_2(\mathbb{C}),\tau)}
=f_{m}(s)\leq f_{m}(m)
\ov{\eqref{quelques-calculs}}{=} 2^{\frac12-\frac1p}m^{\frac12}
\ov{\eqref{quelques-calculs}}{=}
2^{\frac12-\frac1p}f_{m}(0) \\
&\leq
2^{\frac12-\frac1p}f_{m}(r)
=2^{\frac12-\frac1p}
\norm{x}_{\L^p(\Sym_2(\mathbb{C}))}.         
\end{align*}
It remains to prove optimality. Consider the element $
x_0
\ov{\mathrm{def}}{=}
\sigma_1+\i \sigma_2$. Then $
\alpha=0$, $\beta=1$, $\gamma=\i$. So we have $
m \ov{\eqref{def-de-m}}{=} 2$, $
u \ov{\eqref{def-de-m}}{=} 0$, $v \ov{\eqref{def-de-m}}{=} 0$, $w=2$, $
r \ov{\eqref{def-de-r}}{=} 0$ and $
s \ov{\eqref{eq:def-s-spin}}{=} 2$. Hence $
\norm{x_0}_{\L^p(\Sym_2(\mathbb{C}))}
=
f_{2}(0)
=
\sqrt{2}$, while $
\norm{x_0}_{\L^{p,D}(\M_2(\mathbb{C}),\tau)}
=
f_{2}(2)
=
2^{1-\frac1p}$. Therefore
\[
\frac{\norm{x_0}_{\L^p(\Sym_2(\mathbb{C}))}}
{\norm{x_0}_{\L^{p,D}(\M_2(\mathbb{C}),\tau)}}
=
2^{\frac1p-\frac12}
\quad \text{and} \quad 
\frac{\norm{x_0}_{\L^{p,D}(\M_2(\mathbb{C}),\tau)}}
{\norm{x_0}_{\L^p(\Sym_2(\mathbb{C}))}}
=
2^{\frac12-\frac1p}.
\]
Thus the constants are optimal.
\end{proof}

\begin{prop}
\label{prop-spin-comparison-bis}
Suppose that $1 \leq p < \infty$ and $p \ne 2$. For any $x \in \Sym_2(\mathbb{C})$, the two norms $\norm{x}_{\L^p(\Sym_2(\mathbb{C}))}$ and $\norm{x}_{\L^{p,D}(\M_2(\mathbb{C}),\tau)}$ coincide if and only if the matrix $x$ is normal in $\M_2(\mathbb{C})$.
\end{prop}

\begin{proof}
Consider any element $x=\alpha \I+\beta \sigma_1+\gamma \sigma_2$ with $\alpha,\beta,\gamma \in \mathbb{C}$. By a computation similar to \eqref{egalite-sans-fin-4}, we have
\begin{equation}
\label{calcul-x-x*}
xx^*
=(\alpha \I+\beta \sigma_1+\gamma \sigma_2)(\alpha \I+\beta \sigma_1+\gamma \sigma_2)^*
=m\I+u\sigma_1+v\sigma_2-w\sigma_3.
\end{equation}
Hence $
x^*x-xx^*
\ov{\eqref{eq:x*x-spin} \eqref{calcul-x-x*}}{=} 
2w\sigma_3$, where $w=2\Im(\ovl{\beta}\gamma)$. We infer that
\[
x^*x=xx^*
\quad \Longleftrightarrow \quad
w=0
\quad \Longleftrightarrow \quad
\Im(\ovl{\beta}\gamma)=0.
\]
This proves that the matrix $x$ is normal if and only if $\Im(\ovl{\beta}\gamma)=0$. Since the two norms are given by \eqref{une-autre-egalite} and \eqref{eq:Dixmier-spin}, they coincide if and only if $r=s$ (since the function \eqref{def-de-fm} is strictly increasing or strictly decreasing). Since $s^2-r^2 \ov{\eqref{def-de-r} \eqref{eq:def-s-spin}}{=} w^2$, we have $r=s$ if and only if $w=0$, hence if and only if $x$ is normal.
\end{proof}

\begin{remark} \normalfont
\label{remark-compatibility-spin}
The formula \eqref{une-autre-egalite} obtained in Example \ref{ex-sym-2x2} is of course compatible with \eqref{eq:Lp-general-spin}. Indeed, consider an element of the abstract spin factor $\cal{V}$ of the form $x=\lambda 1+h$, where $\lambda \in \mathbb{C}$, $h = s e_1+t e_2$, with $s,t \in \mathbb{R}$. We have 
\begin{equation}
\label{une-autre-egalite-bis}
\norm{h}_2
=\sqrt{s^2+t^2}.
\end{equation}
Under the Jordan $*$-isomorphism $J \co \cal V \to \Sym_2(\mathbb{C})$ the element $x$ corresponds to $
J(x)=\lambda \I+s\sigma_1+t\sigma_2$. Hence, in the notation of Example \ref{ex-sym-2x2}, we have $\alpha=\lambda$, $\beta=s$ and $\gamma=t$. If we set $\rho \ov{\mathrm{def}}{=} \norm{h}_2$, then
\begin{equation}
\label{1234ERT}
m
\ov{\eqref{def-de-m}}{=} |\lambda|^2 + s^2 + t^2
\ov{\eqref{une-autre-egalite-bis}}{=} |\lambda|^2 + \norm{h}_2^2
=|\lambda|^2+\rho^2
\end{equation}
and
\begin{equation}
\label{446YTRR}
r
\ov{\eqref{def-de-r}}{=}
2\sqrt{\Re(\ovl{\lambda}s)^2+\Re(\ovl{\lambda}t)^2}
=2\sqrt{(\Re\ovl{\lambda})^2 s^2+(\Re \ovl{\lambda})^2 t^2}
\ov{\eqref{une-autre-egalite-bis}}{=} 2|\Re(\lambda)| \norm{h}_2
=2|\Re(\lambda)| \rho.
\end{equation}
It follows that the unordered pair $\{m-r,m+r\}$ coincides with $
\{|\lambda-\rho|^2, |\lambda+\rho|^2\}$. Indeed, if $\Re(\lambda) \geq 0$, then
\begin{equation}
\label{ine-fin-33}
m \pm r
\ov{\eqref{1234ERT} \eqref{446YTRR}}{=} |\lambda|^2+\rho^2 \pm 2|\Re(\lambda)| \rho
=|\lambda|^2+\rho^2 \pm 2\Re(\lambda \rho)
=|\lambda \pm \rho|^2,
\end{equation}
whereas if $\Re(\lambda) \leq 0$, the two terms are interchanged. Consequently, we have
\begin{equation}
\label{une-autre-egalite-bis-2}
\frac{(m-r)^{\frac p2}+(m+r)^{\frac p2}}{2}
\ov{\eqref{ine-fin-33}}{=}
\frac{|\lambda-\rho|^p+|\lambda+\rho|^p}{2}.
\end{equation}
Hence the formula of Example \ref{ex-sym-2x2} becomes
\[
\norm{x}_{\L^p(\cal M)}
\ov{\eqref{une-autre-egalite}}{=}\bigg(
\frac{(m-r)^{\frac{p}{2}}+(m+r)^{\frac{p}{2}}}{2}
\bigg)^{\frac{1}{p}}
\ov{\eqref{une-autre-egalite-bis-2}}{=} \bigg(\frac{|\lambda+\norm{h}_2|^p+|\lambda-\norm{h}_2|^p}{2}\bigg)^{\frac1p},
\]
which is exactly \eqref{eq:Lp-general-spin}.
\end{remark}


\begin{example} \normalfont
\label{Ex-dim-4}
Recall that the Jordan algebra $\M_2(\mathbb{C})$ (equipped with the Jordan product defined in \eqref{Jordan-product}) is isomorphic to the complex spin factor of dimension $4$. More precisely, consider the matrices
\begin{equation}
\label{Pauli-four}
\I
\ov{\mathrm{def}}{=}
\begin{bmatrix}
1 & 0\\
0 & 1
\end{bmatrix},
\qquad
\sigma_1
\ov{\mathrm{def}}{=}
\begin{bmatrix}
1 & 0\\
0 & -1
\end{bmatrix},
\qquad
\sigma_2
\ov{\mathrm{def}}{=}
\begin{bmatrix}
0 & 1\\
1 & 0
\end{bmatrix},
\qquad
\sigma_3
\ov{\mathrm{def}}{=}
\begin{bmatrix}
0 & \i\\
-\i & 0
\end{bmatrix}.
\end{equation}
These matrices form a basis of the vector space $\M_2(\mathbb{C})$ and satisfy
\begin{equation}
\label{Pauli-relations-bis}
\sigma_j^*=\sigma_j,
\qquad
\sigma_j \circ \sigma_j=\I,
\qquad
\sigma_j \circ \sigma_k=0
\quad \text{for } j \not= k.
\end{equation}
Let $
\cal{V}
\ov{\mathrm{def}}{=}
\mathbb{C}1 \oplus \mathbb{C}e_1 \oplus \mathbb{C}e_2 \oplus \mathbb{C}e_3$ be the abstract complex spin factor of dimension $4$, where $\{1,e_1,e_2,e_3\}$ is an orthonormal basis of the underlying Hilbert space and the conjugation is given by
\[
\overline{\lambda 1 + z_1 e_1 + z_2 e_2 + z_3 e_3}
\ov{\mathrm{def}}{=}
\overline{\lambda}1-\overline{z_1}e_1-\overline{z_2}e_2-\overline{z_3}e_3.
\]
In particular, we have
\[
e_j^*=-\overline{e_j}=e_j,
\qquad
j=1,2,3.
\]
Hence the elements $e_1,e_2,e_3$ are selfadjoint. Moreover, we have
\[
e_j \circ e_j \ov{\eqref{product-complex-spin-factor}}{=} 1
\quad \text{and} \quad
e_j \circ e_k \ov{\eqref{product-complex-spin-factor}}{=} 0
\quad \text{for } j \not= k.
\]
Consequently the linear map $
J \co \cal{V} \to \M_2(\mathbb{C})$, $
\lambda 1 + z_1 e_1 + z_2 e_2 + z_3 e_3
\mapsto
\lambda \I + z_1 \sigma_1 + z_2 \sigma_2 + z_3 \sigma_3$  is an isomorphism of $\JBW^*$-algebras. The associated $\JBW$-algebra is the space $\M_2(\mathbb{C})_{\sa}$ of selfadjoint $2 \times 2$ complex matrices, equipped with the operator norm and the Jordan product. This $\JBW$-algebra is a real spin factor of dimension $4$. We equip $\M_2(\mathbb{C})$ with the normalized trace
\[
\tau(x)
\ov{\mathrm{def}}{=}
\frac{1}{2}\tr(x),
\qquad x \in \M_2(\mathbb{C}).
\]
Then $
\tau(\I)=1$, $\tau(\sigma_1)=\tau(\sigma_2)=\tau(\sigma_3)=0$. Hence this trace coincides with the complexified canonical trace of the abstract spin factor. Consider an element $x$ of $\M_2(\mathbb{C})$ of the form
\begin{equation}
\label{divers-6789}
x
=\alpha \I+\beta_1 \sigma_1+\beta_2 \sigma_2+\beta_3 \sigma_3,
\qquad
\text{where }\alpha,\beta_1,\beta_2,\beta_3 \in \mathbb{C}.
\end{equation}
Since the elements $\I,\sigma_1,\sigma_2$ and $\sigma_3$ are selfadjoint, we have $
x^*
\ov{\eqref{divers-6789}}{=}
\overline{\alpha}\I+\overline{\beta_1}\sigma_1+\overline{\beta_2}\sigma_2+\overline{\beta_3}\sigma_3$. Using the commutation rules \eqref{Pauli-relations-bis}, we obtain
\begin{align}
\MoveEqLeft
\label{inter-infty-67}
x^* \circ x 
=(\overline{\alpha}\I+\overline{\beta_1}\sigma_1+\overline{\beta_2}\sigma_2+\overline{\beta_3}\sigma_3)
\circ
(\alpha \I+\beta_1 \sigma_1+\beta_2 \sigma_2+\beta_3 \sigma_3) \\
&=
\overline{\alpha}\alpha (\I \circ \I)
+\overline{\alpha}\beta_1 (\I \circ \sigma_1)
+\overline{\alpha}\beta_2 (\I \circ \sigma_2)
+\overline{\alpha}\beta_3 (\I \circ \sigma_3) \nonumber\\
&\quad
+\overline{\beta_1}\alpha (\sigma_1 \circ \I)
+\overline{\beta_1}\beta_1 (\sigma_1 \circ \sigma_1)
+\overline{\beta_1}\beta_2 (\sigma_1 \circ \sigma_2)
+\overline{\beta_1}\beta_3 (\sigma_1 \circ \sigma_3) \nonumber\\
&\quad
+\overline{\beta_2}\alpha (\sigma_2 \circ \I)
+\overline{\beta_2}\beta_1 (\sigma_2 \circ \sigma_1)
+\overline{\beta_2}\beta_2 (\sigma_2 \circ \sigma_2)
+\overline{\beta_2}\beta_3 (\sigma_2 \circ \sigma_3) \nonumber\\
&\quad
+\overline{\beta_3}\alpha (\sigma_3 \circ \I)
+\overline{\beta_3}\beta_1 (\sigma_3 \circ \sigma_1)
+\overline{\beta_3}\beta_2 (\sigma_3 \circ \sigma_2)
+\overline{\beta_3}\beta_3 (\sigma_3 \circ \sigma_3) \nonumber\\
&\ov{\eqref{Pauli-relations-bis}}{=}
|\alpha|^2 \I
+\overline{\alpha}\beta_1 \sigma_1
+\overline{\alpha}\beta_2 \sigma_2
+\overline{\alpha}\beta_3 \sigma_3 
+\overline{\beta_1}\alpha \sigma_1
+|\beta_1|^2 \I
+0
+0 \nonumber\\
&\quad
+\overline{\beta_2}\alpha \sigma_2
+0
+|\beta_2|^2 \I
+0 
+\overline{\beta_3}\alpha \sigma_3
+0
+0
+|\beta_3|^2 \I \nonumber\\
&=
\bigl(|\alpha|^2+|\beta_1|^2+|\beta_2|^2+|\beta_3|^2\bigr)\I 
+\bigl(\overline{\alpha}\beta_1+\overline{\beta_1}\alpha\bigr)\sigma_1
+\bigl(\overline{\alpha}\beta_2+\overline{\beta_2}\alpha\bigr)\sigma_2
+\bigl(\overline{\alpha}\beta_3+\overline{\beta_3}\alpha\bigr)\sigma_3 \nonumber\\
&=\bigl(|\alpha|^2+|\beta_1|^2+|\beta_2|^2+|\beta_3|^2\bigr)\I 
+2\Re(\overline{\alpha}\beta_1)\sigma_1
+2\Re(\overline{\alpha}\beta_2)\sigma_2
+2\Re(\overline{\alpha}\beta_3)\sigma_3. \nonumber
\end{align}
Set
\begin{equation}
\label{eq:def-muj-spin4}
m
\ov{\mathrm{def}}{=}
|\alpha|^2+|\beta_1|^2+|\beta_2|^2+|\beta_3|^2,
\qquad
u_j
\ov{\mathrm{def}}{=}
2\Re(\overline{\alpha}\beta_j),
\quad j=1,2,3,
\end{equation}
and
\begin{equation}
\label{eq:def-r-spin4}
r
\ov{\mathrm{def}}{=}
\sqrt{u_1^2+u_2^2+u_3^2}.
\end{equation}
Then
\begin{equation}
\label{inter-23456}
x^* \circ x
\ov{\eqref{inter-infty-67}\eqref{eq:def-muj-spin4}}{=}
m\I+u_1\sigma_1+u_2\sigma_2+u_3\sigma_3
=
\begin{bmatrix}
m+u_1 & u_2+\i u_3\\
u_2-\i u_3 & m-u_1
\end{bmatrix}.
\end{equation}
Hence the characteristic polynomial of $x^* \circ x$ is
\begin{align*}
\MoveEqLeft
\chi_{x^* \circ x}(t)
\ov{\eqref{inter-23456}}{=}
\det\begin{bmatrix}
m+u_1-t & u_2+\i u_3\\
u_2-\i u_3 & m-u_1-t
\end{bmatrix} 
=(m+u_1-t)(m-u_1-t)-(u_2+\i u_3)(u_2-\i u_3) \\
&= (m-t)^2-u_1^2-u_2^2-u_3^2 
\ov{\eqref{eq:def-r-spin4}}{=} (m-t)^2-r^2.
\end{align*}
Therefore the eigenvalues of $x^* \circ x$ are $m-r$ and $m+r$. Since $x^* \circ x$ is positive, both numbers are positive. Consequently, we have
\begin{equation}
\label{eq:Lp-spin4}
\norm{x}_{\L^p(\M_2(\mathbb{C}))}
\ov{\eqref{norm-LpNA}}{=}
\big(\tau \big[(x^* \circ x)^{\frac{p}{2}}\big]\big)^{\frac{1}{p}}
=\big(\tfrac{1}{2}\Tr \big[(x^* \circ x)^{\frac{p}{2}}\big]\big)^{\frac{1}{p}}
=\bigg(
\frac{(m-r)^{\frac p2}+(m+r)^{\frac p2}}{2}
\bigg)^{\frac1p}.
\end{equation}

Now, we compute the Dixmier norm. Using the equalities $\sigma_1^2=\sigma_2^2=\sigma_3^2=\I$, $\sigma_1\sigma_2=-\i \sigma_3$, 
$\sigma_2\sigma_1=\i \sigma_3$, $\sigma_1\sigma_3=\i \sigma_2$, $\sigma_3\sigma_1=-\i \sigma_2$, $
\sigma_2\sigma_3=-\i \sigma_1$, $\sigma_3\sigma_2=\i \sigma_1$ and expanding term by term, we get
\begin{align*}
\MoveEqLeft
x^*x 
=(\overline{\alpha}\I+\overline{\beta_1}\sigma_1+\overline{\beta_2}\sigma_2+\overline{\beta_3}\sigma_3)
(\alpha \I+\beta_1 \sigma_1+\beta_2 \sigma_2+\beta_3 \sigma_3) \\
&=
|\alpha|^2 \I
+\overline{\alpha}\beta_1 \sigma_1
+\overline{\alpha}\beta_2 \sigma_2
+\overline{\alpha}\beta_3 \sigma_3 
+\overline{\beta_1}\alpha \sigma_1
+\overline{\beta_1}\beta_1 \sigma_1^2
+\overline{\beta_1}\beta_2 \sigma_1\sigma_2
+\overline{\beta_1}\beta_3 \sigma_1\sigma_3 \\
&\quad
+\overline{\beta_2}\alpha \sigma_2
+\overline{\beta_2}\beta_1 \sigma_2\sigma_1
+\overline{\beta_2}\beta_2 \sigma_2^2
+\overline{\beta_2}\beta_3 \sigma_2\sigma_3 
+\overline{\beta_3}\alpha \sigma_3
+\overline{\beta_3}\beta_1 \sigma_3\sigma_1
+\overline{\beta_3}\beta_2 \sigma_3\sigma_2
+\overline{\beta_3}\beta_3 \sigma_3^2 \\
&=
|\alpha|^2 \I
+\overline{\alpha}\beta_1 \sigma_1
+\overline{\alpha}\beta_2 \sigma_2
+\overline{\alpha}\beta_3 \sigma_3 
+\overline{\beta_1}\alpha \sigma_1
+|\beta_1|^2 \I
+\overline{\beta_1}\beta_2 (-\i \sigma_3)
+\overline{\beta_1}\beta_3 (\i \sigma_2) \\
&\quad
+\overline{\beta_2}\alpha \sigma_2
+\overline{\beta_2}\beta_1 (\i \sigma_3)
+|\beta_2|^2 \I
+\overline{\beta_2}\beta_3 (-\i \sigma_1) 
+\overline{\beta_3}\alpha \sigma_3
+\overline{\beta_3}\beta_1 (-\i \sigma_2)
+\overline{\beta_3}\beta_2 (\i \sigma_1)
+|\beta_3|^2 \I.
\end{align*}
The coefficient of $\I$ is $
m
\ov{\eqref{eq:def-muj-spin4}}{=}
|\alpha|^2+|\beta_1|^2+|\beta_2|^2+|\beta_3|^2$. Moreover, the coefficient of $\sigma_1$ is
\[
c_1
\ov{\mathrm{def}}{=} \overline{\alpha}\beta_1+\overline{\beta_1}\alpha
-\i \overline{\beta_2}\beta_3+\i \overline{\beta_3}\beta_2
=
2\Re(\alpha\overline{\beta_1})
+2\Im(\overline{\beta_2}\beta_3).
\]
The coefficient of $\sigma_2$ is given by
\[
c_2
\ov{\mathrm{def}}{=} \overline{\alpha}\beta_2+\overline{\beta_2}\alpha
+\i \overline{\beta_1}\beta_3-\i \overline{\beta_3}\beta_1
=
2\Re(\alpha\overline{\beta_2})
-2\Im(\overline{\beta_1}\beta_3).
\]
Finally, the coefficient of $\sigma_3$ is
\[
c_3
\ov{\mathrm{def}}{=} \overline{\alpha}\beta_3+\overline{\beta_3}\alpha
-\i \overline{\beta_1}\beta_2+\i \overline{\beta_2}\beta_1
=
2\Re(\alpha\overline{\beta_3})
+2\Im(\overline{\beta_1}\beta_2).
\]
Using the explicit matrices, we obtain
\[
x^*x
= m\I+c_1\sigma_1+c_2\sigma_2+c_3\sigma_3
\ov{\eqref{Pauli-four}}{=}
\begin{bmatrix}
m+c_1 & c_2+\i c_3\\
c_2-\i c_3 & m-c_1
\end{bmatrix}.
\]
Consequently, the characteristic polynomial is given by
\begin{align*}
\chi_{x^*x}(t)
&\ov{\mathrm{def}}{=}
\det \begin{bmatrix}
m+c_1 -t & c_2+\i c_3\\
c_2-\i c_3 & m-c_1 -t
\end{bmatrix} 
= (m+c_1-t)(m-c_1-t)-(c_2+\i c_3)(c_2-\i c_3) \\
&= (t-m)^2-c_1^2-(c_2^2+c_3^2) 
= (t-m)^2-(c_1^2+c_2^2+c_3^2).
\end{align*}
If we set $
s
\ov{\mathrm{def}}{=}
\sqrt{c_1^2+c_2^2+c_3^2}$ then $
\chi_{x^*x}(t)=(t-m)^2-s^2$. It follows that the eigenvalues of $x^*x$ are $m-s$ and $m+s$. Since $x^*x$ is positive in $\M_2(\mathbb{C})$, both numbers $m-s$ and $m+s$ are nonnegative. Therefore
\begin{equation}
\label{eq:Dixmier-spin4}
\norm{x}_{\L^{p,D}(\M_2(\mathbb{C}),\tau)}
\ov{\eqref{Dixmier-NC-spaces}}{=}
\big(\tau(|x|^p)\big)^{\frac1p}
=
\bigg(
\frac{(m-s)^{\frac p2}+(m+s)^{\frac p2}}{2}
\bigg)^{\frac1p}.
\end{equation}
\end{example}

\begin{example} \normalfont
\label{ex-H2H-spin} 
Let $\mathbb{H}$ be the skew field of quaternions, with standard basis $(1,\i,\j,\k)$ and conjugation
\[
\overline{a+b\,\i+c\,\j+d\,\k}
\ov{\mathrm{def}}{=}
a-b\,\i-c\,\j-d\,\k.
\]
Consider the real Jordan algebra $
\H_2(\mathbb{H})
\ov{\mathrm{def}}{=}
\bigg\{
\begin{bmatrix}
a & q \\
\ovl{q} & b
\end{bmatrix}
: a,b \in \R, \ q \in \mathbb{H}
\bigg\}$, endowed with the Jordan product introduced in \eqref{Jordan-product}. We introduce the following elements of $\H_2(\mathbb{H})$:
\[
\I
\ov{\mathrm{def}}{=}
\begin{bmatrix}
1 & 0\\
0 & 1
\end{bmatrix},
\qquad
s_1
\ov{\mathrm{def}}{=}
\begin{bmatrix}
1 & 0\\
0 & -1
\end{bmatrix},
\qquad
s_2
\ov{\mathrm{def}}{=}
\begin{bmatrix}
0 & 1\\
1 & 0
\end{bmatrix},
\]
\[
s_3
\ov{\mathrm{def}}{=}
\begin{bmatrix}
0 & \i\\
-\i & 0
\end{bmatrix},
\qquad
s_4
\ov{\mathrm{def}}{=}
\begin{bmatrix}
0 & \j\\
-\j & 0
\end{bmatrix},
\qquad
s_5
\ov{\mathrm{def}}{=}
\begin{bmatrix}
0 & \k\\
-\k & 0
\end{bmatrix}.
\]
These elements form a basis of the real vector space $\H_2(\mathbb{H})$. Moreover, they are selfadjoint and satisfy
\begin{equation}
\label{eq-quaternionic-spin-relations}
s_\ell \circ s_\ell=\I,
\qquad
s_\ell \circ s_m=0
\quad \text{for } \ell \not= m.
\end{equation}
Indeed, the first identity is immediate from $\i^2=\j^2=\k^2=-1$, and the second follows from the anticommutation relations in $\mathbb{H}$. Therefore $
\H_2(\mathbb{H})
=
\R\I \oplus \R s_1 \oplus \R s_2 \oplus \R s_3 \oplus \R s_4 \oplus \R s_5$ is a real spin factor of dimension $6$. We refer to \cite[Section 6.4]{HOS84} for more explanation.

Let $
\cal{V}
\ov{\mathrm{def}}{=}
\mathbb{C}1 \oplus \mathbb{C}e_1 \oplus \mathbb{C}e_2 \oplus \mathbb{C}e_3 \oplus \mathbb{C}e_4 \oplus \mathbb{C}e_5$ be the abstract complex spin factor of dimension $6$, where $\{1,e_1,\ldots,e_5\}$ is an orthonormal basis of the underlying Hilbert space and the conjugation is given by
\[
\overline{\lambda 1 + z_1 e_1 + \cdots + z_5 e_5}
\ov{\mathrm{def}}{=}
\overline{\lambda}1-\overline{z_1}e_1-\cdots-\overline{z_5}e_5.
\]
Then the linear map $
J \co \cal V \to \H_2(\mathbb{H})_{\mathbb{C}}$, 
$\lambda 1 + z_1 e_1 + \cdots + z_5 e_5
\mapsto
\lambda \I + z_1 s_1 + \cdots + z_5 s_5$ is an isomorphism of $\JBW^*$-algebras onto the complexification of $\H_2(\mathbb{H})$. We equip $\H_2(\mathbb{H})$ with the normalized trace
\[
\tau
\left(
\begin{bmatrix}
a & q\\
\overline{q} & b
\end{bmatrix}
\right)
\ov{\mathrm{def}}{=}
\frac{a+b}{2},
\]
and we still denote by $\tau$ its complex-linear extension to $\H_2(\mathbb{H})_{\mathbb{C}}$. Then
\[
\tau(\I)=1
\qquad \text{and} \quad
\tau(s_\ell)=0,
\quad \ell=1,\ldots,5.
\]
Hence this trace coincides with the complexified canonical trace of the abstract spin factor.
\end{example}

\begin{prop}
\label{prop-H2H-Lp}
Suppose that $1 \leq p < \infty$. Let $
x=\lambda \I+h$ be an element of $\H_2(\mathbb{H})_{\mathbb{C}}$, where $\lambda \in \mathbb{C}$ and
\[
h=t_1s_1+t_2s_2+t_3s_3+t_4s_4+t_5s_5
\]
with $t_1,\ldots,t_5 \in \mathbb{R}$. Set $
\rho
\ov{\mathrm{def}}{=}
\sqrt{t_1^2+t_2^2+t_3^2+t_4^2+t_5^2}$. Then
\begin{equation}
\label{eq-H2H-Lp}
\norm{x}_{\L^p(\H_2(\mathbb{H})_{\mathbb{C}})}
=
\bigg(
\frac{|\lambda+\rho|^p+|\lambda-\rho|^p}{2}
\bigg)^{\frac{1}{p}}.
\end{equation}
\end{prop}

\begin{proof}
If $h=0$, then $x=\lambda \I$ and the formula is obvious. Assume now that $h \not= 0$ and set $
e
\ov{\mathrm{def}}{=}
\frac{h}{\rho}$. By \eqref{eq-quaternionic-spin-relations}, the element $e$ is selfadjoint, orthogonal to $\I$, and satisfies $
e \circ e=\I$. Hence the Jordan $*$-subalgebra generated by $\{\I,e\}$ is a two-dimensional complex spin factor. Since
$
x=\lambda \I+\rho e$, 
Proposition~\ref{prop-two-dimensional} applied to the pair $(\I,e)$ yields
\[
\norm{x}_{\L^p(\H_2(\mathbb{H})_{\mathbb{C}})}
=
\bigg(
\frac{|\lambda+\rho|^p+|\lambda-\rho|^p}{2}
\bigg)^{\frac{1}{p}}.
\]
\end{proof}

\begin{example} \normalfont
\label{ex-H2O-spin}
We first recall some information on the algebra $\O$ of octonions. 
Recall that $\O$ is an eight-dimensional nonassociative real division algebra with basis $(1,e_1,e_2,e_3,e_4,e_5,e_6,e_7)$, whose multiplication is given in Figure~\ref{Table-multiplication-octonions}. More precisely, the entry in the $i$-th row and $j$-th column is the product of the basis element indexing the row by the basis element indexing the column.

\begin{figure}[ht]  
\centering
{\small   
\begin{tabular}{|c|c|c|c|c|c|c|c|c|}                    \hline   
      & $e_1$ & $e_2$ & $e_3$ & $e_4$  & $e_5$ & $e_6$ & $e_7$ \\ \hline   
  $e_1$ & $-1$  & $e_4$ & $e_7$ & $-e_2$ & $e_6$ & $-e_5$ & $-e_3$ \\ \hline   
$e_2$ & $-e_4$ & $-1$ & $e_5$ & $e_1$ & $-e_3$ & $e_7$ & $-e_6$     \\ \hline   
$e_3$ & $-e_7$ & $-e_5$ & $-1$ & $e_6$ & $e_2$ & $-e_4$ & $e_1$   \\ \hline   
$e_4$ & $e_2$ & $-e_1$ & $-e_6$ & $-1$ & $e_7$ & $e_3$ & $-e_5$   \\ \hline   
$e_5$ & $-e_6$ & $e_3$ & $-e_2$ & $-e_7$ & $-1$ & $e_1$ & $e_4$    \\ \hline   
$e_6$ & $e_5$ & $-e_7$ & $e_4$ & $-e_3$ & $-e_1$ & $-1$ & $e_2$     \\ \hline   
$e_7$ & $e_3$ & $e_6$ & $-e_1$ & $e_5$ & $-e_4$ & $-e_2$ & $-1$    \\  \hline   
\end{tabular}}    
\caption{Octonion multiplication table}
\label{Table-multiplication-octonions}
\end{figure}

Each octonion $x \in \O$ can be written as a linear combination
$$
x 
= x_0 + x_1 e_1 + x_2 e_2+ x_3 e_3+ x_4 e_4+ x_5 e_5+ x_6 e_6  + x_7 e_7, \qquad x_0,x_1, x_2,x_3,x_4,x_5,x_6,x_7 \in \R
$$
of the basis elements $1,e_1,e_2,e_3,e_4,e_5,e_6,e_7$. The conjugation is the operation that flips the signs of all the <<imaginary>> coefficients
$$
\ovl{x} 
\ov{\mathrm{def}}{=} x_0 - x_1 e_1 - x_2 e_2- x_3 e_3- x_4 e_4- x_5 e_5- x_6 e_6 - x_7 e_7.
$$
It is an involution that reverses the order of multiplication, that is, for any $x,y \in\mathbb{O}$, we have $\ovl{x y}=\ovl{y} \ovl{x}$. We refer to the paper \cite{Bae02} and to the book \cite{CoS03} for more information. 

Consider the real Jordan algebra $
\A
\ov{\mathrm{def}}{=}
\H_2(\O)
=
\bigg\{
\begin{bmatrix}
a & x\\
\overline{x} & b
\end{bmatrix}
: a,b \in \R,\ x \in \O
\bigg\}$, equipped with the Jordan product defined in \eqref{Jordan-product}. We introduce the following elements of $\A$:
\[
\I
\ov{\mathrm{def}}{=}
\begin{bmatrix}
1 & 0\\
0 & 1
\end{bmatrix},
\qquad
s_1
\ov{\mathrm{def}}{=}
\begin{bmatrix}
1 & 0\\
0 & -1
\end{bmatrix},
\qquad
s_2
\ov{\mathrm{def}}{=}
\begin{bmatrix}
0 & 1\\
1 & 0
\end{bmatrix}
\quad \text{and} \quad
s_{k+2}
\ov{\mathrm{def}}{=}
\begin{bmatrix}
0 & e_k\\
-e_k & 0
\end{bmatrix}
\]
for any integer $k \in \{1,\ldots,7\}$. Then $
\A
=\R \I \oplus \R s_1 \oplus \R s_2 \oplus \R s_3 \oplus \cdots \oplus \R s_9$. Moreover, all these elements are selfadjoint and satisfy
\begin{equation}
\label{eq-octonionic-spin-relations}
s_\ell \circ s_\ell=\I,
\qquad
s_\ell \circ s_m=0
\quad \text{for } \ell \not= m.
\end{equation}
Indeed, the first identity is immediate, while the second follows from the fact that the imaginary units of $\O$ anticommute and square to $-1$. Therefore $\A$ is a real spin factor of dimension $10$ (see also \cite[Lemma A.14]{GKR24}). Let
$
\cal V
\ov{\mathrm{def}}{=}
\mathbb{C} 1 \oplus \mathbb{C} u_1 \oplus \cdots \oplus \mathbb{C} u_9$ be the abstract complex spin factor of dimension $10$, where $\{1,u_1,\ldots,u_9\}$ is an orthonormal basis of the underlying Hilbert space and the conjugation is given by
\[
\overline{\lambda 1 + z_1u_1+\cdots+z_9u_9}
\ov{\mathrm{def}}{=}
\overline{\lambda}1-\overline{z_1}u_1-\cdots-\overline{z_9}u_9.
\]
Then the linear map $
J \co \cal{V} \to \A+\i \A$, $
\lambda 1 + z_1u_1+\cdots+z_9u_9
\mapsto
\lambda \I + z_1s_1+\cdots+z_9s_9$ is an isomorphism of $\JBW^*$-algebras. We equip $\A$ with the normalized trace
\[
\tau
\left(
\begin{bmatrix}
a & x\\
\overline{x} & b
\end{bmatrix}
\right)
\ov{\mathrm{def}}{=}
\frac{a+b}{2},
\]
and we still denote by $\tau$ its complex-linear extension to $\A+\i \A$. Then $\tau(\I)=1$ and $\tau(s_\ell)=0$ for any integer $\ell=1,\ldots,9$. Hence this trace coincides with the complexified canonical trace of the abstract spin factor.
\end{example}

The same argument as in Proposition \ref{prop-H2H-Lp} gives the following result.

\begin{prop}
\label{prop-H2O-Lp}
Suppose that $1 \leq p < \infty$. Let $x=\lambda \I+h$ be an element of $\A+\i \A$, where $\lambda \in \mathbb{C}$ and
$
h=t_1s_1+\cdots+t_9s_9$ with $t_1,\ldots,t_9 \in \R$. Set $
\rho
\ov{\mathrm{def}}{=}
\sqrt{t_1^2+\cdots+t_9^2}$. Then
\begin{equation}
\label{eq-H2O-Lp}
\norm{x}_{\L^p(\A+\i \A)}
=
\bigg(
\frac{|\lambda+\rho|^p+|\lambda-\rho|^p}{2}
\bigg)^{\frac{1}{p}}.
\end{equation}
\end{prop}

\section{The Albert algebra $\mathrm{H}_3(\O)$ and its complexification $\mathrm{H}_3(\O_\mathbb{C})$}
\label{sec-complex-octonions}



The aim of this section is to recall the basic structure of the Albert algebra and its complexification, and to derive an explicit formula for $\norm{x}_{\L^p(\cal{M})}
\ov{\mathrm{def}}{=} \big(\tau \big[(x^* \circ x)^{\frac{p}{2}}\big]\big)^{\frac{1}{p}}$ on \textit{selfadjoint} elements. 

\paragraph{The Albert algebra $\mathrm{H}_3(\O)$} 
Here, we use the notations introduced in Example \ref{ex-H2O-spin} for the octonion algebra. The space 
$$
\H_3(\O)
\ov{\mathrm{def}}{=} \left\{\begin{bmatrix}
   a  & \alpha & \beta  \\
   \ovl{\alpha}  & b & \gamma	\\
    \ovl{\beta} & \ovl{\gamma} &  c  \\
\end{bmatrix}: \alpha,\beta,\gamma \in \O, a,b,c \in \R \right\}
$$ 
of Hermitian $3 \times 3$ matrices with entries in the octonion algebra $\O$ equipped with the Jordan product $(x,y) \mapsto x \circ y \ov{\eqref{Jordan-product}}{=} \frac{1}{2}(xy+yx)$ is a 27-dimensional unital formally real Jordan algebra by \cite[Proposition 2.9.2 p.~69]{HOS84}. This algebra is often called the Albert algebra, since Albert proved in \cite{Alb34} that it is exceptional, that is, it is not isomorphic to a Jordan subalgebra of an associative algebra endowed with the Jordan product \eqref{Jordan-product}.

By \cite[Corollary 3.1.7 p.~77]{HOS84} and its proof, we can equip $\mathrm{H}_3(\O)$ with a norm that makes it a $\JB$-algebra. Endowed with this structure, $\mathrm{H}_3(\O)$ is a $\JBW$-factor of type $\I_3$, see \cite[pp.~90-91]{AlS03}.
%


The algebra of complex octonions, also called the algebra of bioctonions, is the complexification $\O_{\mathbb{C}} \ov{\mathrm{def}}{=} \O \ot_{\mathbb{R}} \mathbb{C}$ of the octonion algebra $\O$. Each element $x$ of $\O_{\mathbb{C}}$ can be written in the form $x = x_0 + x_1 e_1 + x_2 e_2 + x_3 e_3 + x_4 e_4 + x_5 e_5 + x_6 e_6  + x_7 e_7$, where $x_0, x_1, x_2,x_3,x_4,x_5,x_6,x_7 \in \mathbb{C}$. The standard octonionic conjugation extends $\mathbb{C}$-linearly to $\O_{\mathbb{C}}$ and is given by
\[
\ovl{x} 
\ov{\mathrm{def}}{=} x_0 - x_1 e_1 - x_2 e_2- x_3 e_3- x_4 e_4- x_5 e_5- x_6 e_6 - x_7 e_7.
\] 
The space 
$$
\H_3(\O_{\mathbb{C}})
\ov{\mathrm{def}}{=} \left\{\begin{bmatrix}
   a  & \alpha & \beta  \\
   \ovl{\alpha}  & b & \gamma	\\
    \ovl{\beta} & \ovl{\gamma} &  c  \\
\end{bmatrix}: \alpha,\beta,\gamma \in \O_{\mathbb{C}}, a,b,c \in \mathbb{C} \right\}
$$ 
of Hermitian $3 \times 3$ matrices with entries in the algebra $\O_{\mathbb{C}}$ of complex octonions equipped with the Jordan product $(x,y) \mapsto x \circ y \ov{\eqref{Jordan-product}}{=} \frac{1}{2}(xy+yx)$ is the $\JBW^*$-factor associated with the $\JBW$-factor $\H_3(\O)$ by \cite[Corollary 5.1.41 p.~15]{CGRP18}. 

By \cite[Proposition 5.25 p.~152]{AlS03}, there exists a unique normalized trace on $\H_3(\O)$, and this trace is  normal and faithful. Moreover, by \cite[p.~151]{AlS03}, it takes the value $\frac13$ on each minimal projection. Using the description \cite[Theorem 8.1 p.~28]{HKP23} \cite[Proposition A.16]{GKR24} of the minimal projections of the algebra $\H_3(\O_{\mathbb{C}})$, it is easy to check that its complex-linear extension $\tau$ to $\H_3(\O_{\mathbb{C}})$ is given by
\[
\tau\left(\begin{bmatrix}
a & \alpha & \beta\\
\overline{\alpha} & b & \gamma\\
\overline{\beta} & \overline{\gamma} & c
\end{bmatrix}\right)
=
\frac13(a+b+c).
\]

\paragraph{Spectral theory}
Let $\cal{M}$ be a finite-dimensional $\JB^*$-algebra. For our purposes, the key point is that selfadjoint elements admit a spectral decomposition into minimal projections, exactly as in the associative matrix setting. Indeed, by \cite[p.~44]{FaK94} (see also \cite[Theorem III.1.1 p.~43]{FaK94} and \cite[Proposition 2.2 p.~6]{HKP23}), every selfadjoint element $x \in \cal{M}$ admits a spectral decomposition of the form
\begin{equation}
\label{spectral-decomposition}
x
=\lambda_1(x) p_1+\dots+\lambda_n(x) p_n,
\end{equation}
where $p_1,\dots,p_n$ are mutually orthogonal minimal projections in $\cal{M}$ belonging to the unital subalgebra generated by $x$, with sum equal to $1$, and $\lambda_1(x),\dots,\lambda_n(x)$ are real numbers. 

\paragraph{The $\L^p$-norm}
Now, we describe the $\L^p$-norm for selfadjoint elements. This means that we give a formula for the norm of nonassociative $\L^p$-spaces of Iochum \cite{Ioc86} associated with the $\JBW$-algebra $\H_3(\O)$. For selfadjoint elements, the nonassociative $\L^p$-norm is simply the $\ell^p$-norm of the Jordan eigenvalues, up to the normalization of the trace. This shows that, on the selfadjoint part, the nonassociative $\L^p$-norm has exactly the expected spectral form. In particular, the exceptional case behaves in complete analogy with the classical Schatten norm on selfadjoint matrices. 

\begin{prop}
\label{prop-selfadjoint-exceptional}
Suppose that $1 \leq p < \infty$. For any selfadjoint element $x \in \H_3(\O_\mathbb{C})$, i.e.~$x \in \H_3(\O)$, we have 
\begin{equation}
\label{norm-Lp-Albert}
\norm{x}_{\L^p(\H_3(\O_\mathbb{C}))}
=\bigg(\frac{|\lambda_1(x)|^p+|\lambda_2(x)|^p+|\lambda_3(x)|^p}{3}\bigg)^{\frac1p}.
\end{equation}
\end{prop}

\begin{proof}
Since $x$ is selfadjoint, it admits a spectral decomposition
\[
x=\lambda_1(x)p_1+\lambda_2(x)p_2+\lambda_3(x)p_3,
\]
where $p_1,p_2,p_3$ are mutually orthogonal minimal projections in the algebra $\H_3(\O_\mathbb{C})$ with $p_1+p_2+p_3=1$. Using the orthogonality relations $p_i \circ p_j=0$ for $i \neq j$ and $p_i \circ p_i=p_i$, we obtain
\[
x^* \circ x 
= x \circ x
=\bigg(\sum_{i=1}^3 \lambda_i(x)p_i\bigg) \circ \bigg(\sum_{j=1}^3 \lambda_j(x)p_j\bigg)
=\sum_{i,j=1}^3 \lambda_i(x)\lambda_j(x) p_i \circ p_j
=\sum_{i=1}^3 \lambda_i(x)^2 p_i.
\]
By the Jordan functional calculus in the commutative unital Jordan subalgebra generated by $x$ and $1$, we get
\begin{equation}
\label{inter-456-bis}
(x^* \circ x)^{\frac p2}
=|\lambda_1(x)|^p p_1 + |\lambda_2(x)|^p p_2 + |\lambda_3(x)|^p p_3.
\end{equation}
Now, since $\tau(p_i)=\frac13$ we obtain
\begin{align}
\MoveEqLeft
\label{inter-AAZDF}
\tau ((x^* \circ x)^{\frac{p}{2}})
\ov{\eqref{inter-456-bis}}{=} \tau\big(|\lambda_1(x)|^p p_1+|\lambda_2(x)|^p p_2+|\lambda_3(x)|^p p_3 \big) \\
&=|\lambda_1(x)|^p \tau(p_1)+|\lambda_2(x)|^p \tau(p_2)+|\lambda_3(x)|^p \tau(p_3) 
=\frac{|\lambda_1(x)|^p+|\lambda_2(x)|^p+|\lambda_3(x)|^p}{3}.   \nonumber      
\end{align}
Consequently, we obtain
\[
\norm{x}_{\L^p(\cal{M})}
\ov{\eqref{norm-LpNA}}{=}
\big(\tau\big[(x^* \circ x)^{\frac{p}{2}}\big]\big)^{\frac{1}{p}}
\ov{\eqref{inter-AAZDF}}{=}
\bigg(\frac{|\lambda_1(x)|^p+|\lambda_2(x)|^p+|\lambda_3(x)|^p}{3}\bigg)^{\frac1p}.
\]
\end{proof}

\begin{remark} \normalfont
A similar formula holds for the $\JBW$-factor $\H_n(\mathbb{H})$, where $\mathbb{H}$ is the division algebra of quaternions. See \cite[Section A.1]{GKR24} for more information on this factor.
\end{remark}

\begin{remark} \normalfont
\label{remark-JBW-triple-bis}
As we said in Remark \ref{remark-JBW-triple}, each $\JBW^*$-algebra $(\cal{M},*,\circ)$ has a canonical structure of $\JBW^*$-triple. The canonical $\JBW^*$-triple associated with $\H_3(\O_{\mathbb{C}})$ is the so-called Cartan factor of type $\VI$. 
\end{remark}

\begin{remark} \normalfont
If we use instead of $\tau$ the trace $\tr$ satisfying $\tr(1)=3$, then
\[
\norm{x}_{\L^p(\cal{M})}
=\bigg(\sum_{i=1}^3 |\lambda_i(x)|^p\bigg)^{\frac1p}.
\]
\end{remark}

\begin{remark} \normalfont
\label{Rodrigues}
Suppose that $1 \leq p < \infty$. Consider an element $x \in \H_3(\O_\mathbb{C})$. The quantity $\big(\sum_{i=1}^3 s_i(x)^p\big)^{\frac{1}{p}}$ appears in \cite[Theorem 4.8]{FeP20}, where $s_1(x),s_2(x),s_3(x)$ are the singular numbers of $x$. The precise relation between this quantity and $\norm{x}_{\L^p(\H_3(\O_\mathbb{C}))}$ and $\norm{x}_{\L^{p,A}(\H_3(\O_\mathbb{C}))}$ is not clear.
\end{remark}


\normalsize



%
%
%

\section{Generalized Probabilistic Theories with $\JBW^*$-algebras}
\label{sec-Generalized}


%
%
%
%
%
%
%
%
%
%
%
%


In this section, we explain why $\JBW$-algebras, and equivalently their associated $\JBW^*$-algebras, provide a natural framework for a large class of generalized probabilistic theories (GPTs). The basic point is that a generalized probabilistic theory is usually formulated in terms of a suitable ordered vector space, while the predual of a $\JBW$-algebra comes with exactly such an order structure, together with a rich spectral theory. In finite dimension, this viewpoint appears in \cite{SAH25}: Euclidean Jordan algebras are among the standard Jordan-algebraic models in GPT theory. Our purpose here is to stress that the $\JBW$ setting is the right von Neumann-type extension of this picture, and that the nonassociative $\L^p$-spaces considered in the present paper fit naturally into it.

We begin by recalling the formalism of generalized probabilistic theories, following the standard references~\cite{DL70}, \cite[Chapter I]{Lam18}, \cite{Lud83}, \cite{Lud85}, \cite{Pla23}, and \cite{Wil25}. Generalized probabilistic theories may be viewed as an axiomatic extension of the probabilistic formalism of quantum mechanics, intended to capture in a unified way the statistical features of physical systems and measurements. Although a substantial part of the literature focuses on the technically more elementary finite-dimensional setting, our aim here is to describe the GPT framework at the level of generality relevant for the present paper. The universality of this formalism is supported by by Ludwig's theorem, stated in \cite[Theorem 1.43, p.~47]{Lam18}, \cite[Theorem 3.1]{Lud83}, and \cite[Theorem 3.7]{Lud85} (see also \cite[p.~249-250]{Cho19} for a variant), which shows that every physical theory satisfying a small number of basic assumptions can be represented as a generalized probabilistic theory. Here, a physical theory is understood as a rule assigning to each state $\omega$ of a fixed physical system and to each measurement event (or effect) $e$ a number $\mu(e,\omega) \in [0,1]$, interpreted as the probability of observing the event $e$ when the system is prepared in the state $\omega$.


\paragraph{Cones and base norm spaces}
We begin by recalling some basic definitions concerning cones and base norm spaces. Consider a real vector space $V$. 
Following \cite[Definition 1.2 p.~2]{AlT07}, we say that a non-empty subset $C$ of $V$ is a cone if	$C+C \subset C$, $\lambda C \subset C$ for all $\lambda \geq 0$ and $C \cap (-C) = \{0\}$. Clearly, every cone is automatically a convex subset. The cone $C$ is said to be generating (or spanning) \cite[Definition 1.5 p.~4]{AlT07} if $V = C - C$, i.e.~if the vector subspace generated by $C$ coincides with the vector space $V$.

Let $C$ be a cone in a real vector space $V$. A linear functional $u \co V \to \R$ is said to be $C$-strictly positive \cite[Definition 1.46 p.~32]{AlT07} whenever $x \in C \backslash \{0\}$ implies $u(x) > 0$. Following \cite[Theorem 1.46 p.~39]{AlT07} (see also \cite[p.~3]{AlS01} and \cite[Definition 1.32 p.~38]{Lam18}), a non-empty convex subset $\cal{B}$ of $C - \{0\}$ is said to be a base for the cone $C$ if for each $x \in C \backslash \{0\}$ there exist $\lambda > 0$ and $b \in \cal{B}$ both uniquely determined such that $x=\lambda b$. According to \cite[Theorem 1.47 p.~40]{AlT07}, a cone $C$ of a vector space $V$ admits a base if and only if $V$ admits a $C$-strictly positive linear functional. 
Indeed, by \cite[Exercise 2 p.~42]{AlT07} (see also \cite[Lemma 1.33 p.~38]{Lam18}), a subset $\cal{B}$ of a cone $C$ of a vector space $V$ is a base for the cone $C$ if and only if there exists a $C$-strictly positive linear functional $u \co V \to \R$ such that 
\begin{equation}
\label{base-eq}
\cal{B}
= \{ x \in C : \la u,x \ra = 1 \}.
\end{equation}
A cone $C$ induces, by \cite[p.~3]{AlT07}, a partial order on the vector space $V$ by setting $x \leq y$ if and only if $y - x \in C$. In this case, $(V,C)$, or simply $V$, is called an ordered real vector space and we say that $C$ is the positive cone.

Now, we recall the notion of Minkowski functional. Let $V$ be a real vector space and let $S$ be a subset of $V$. Then the Minkowski functional of $S$ is the function $\norm{\cdot}_S \co V \to \R_+ \cup \{\infty\}$ defined by
\[
\norm{x}_S 
\ov{\mathrm{def}}{=} \inf\{ r > 0 : x \in rS \},  \quad x \in V.
\]

Recall that a subset $S$ of the space $V$ is said to be absolutely convex if $0 \in S$ and $\lambda x + \mu y \in S$ for all $x,y \in S$ and $\lambda,\mu \in \R$ with $|\lambda| + |\mu| \leq 1$ and absorbent if for all $x \in V$ there exists $r > 0$ such that $x \in rS$. If $S$ is an absorbent absolutely convex subset, then by \cite[p.~11]{Fur16} the Minkowski functional $\norm{\cdot}_S$ is a seminorm. Moreover, the result \cite[Lemma 0.1.5 p.~11]{Fur16} says that $\norm{\cdot}_S$ is a norm if and only if $S$ is radially bounded, i.e.~contains no line through the origin. 

Following \cite[Definition 7.2.8 p.~231]{Cho19}, a semi-base-norm space $(V,C,u)$ consists of an ordered real vector space $(V,C)$ such that $C$ is generating with a specified $C$-strictly positive linear functional $u \co V \to \R$. In this case, $\{x \in C : u(x) = 1\}$ is a base of the cone $C$, denoted by $\cal{B}$.

By \cite[Lemma 7.2.22 p.~238]{Cho19}, the absolutely convex hull of the base $\cal{B}$, that is, the smallest absolutely convex set containing $\cal{B}$, is given by $\absco(\cal{B}) = \{0\}$ if $V =  \{0\}$ and 
$$
\absco(\cal{B}) 
= \{x-y : x, y \in V_+ 
\quad \text{and} \quad u(x) + u(y) = 1\}
$$
if $V \not=  \{0\}$. If $S \ov{\mathrm{def}}{=} \absco(\cal{B})$ is absorbent, the Minkowski functional $\norm{\cdot}_S$ is a seminorm on the vector space $V$, which we call the base seminorm on $V$. We will use this terminology even if $S$ is not absorbent.

The base seminorm is not necessarily a norm, which motivates the following definition. A pre-base-norm space \cite[Definition 7.2.23 p.~238]{Cho19} \cite[p.~88]{Fur16} is a semi-base-norm space $(V,C,u)$ such that $\absco(\cal{B})$ is radially bounded. By \cite[Lemma 2.2.3 p.~890]{Fur16}, this implies that $\absco(\cal{B})$ is absorbent and consequently it is equivalent to say that the base seminorm is a norm. Then the base seminorm $\norm{\cdot}_S$ is called the base norm on $V$. 

A Banach base-norm space \cite[Definition 7.2.23 p.~238]{Cho19} is a pre-base-norm space $(V,C,u)$ such that $V$ is complete with respect to the base-norm and where the positive cone $C$ is closed with respect to the base norm. We refer to \cite{Alf71}, \cite{Cho19} and \cite{Fur16} for more information on base norm spaces. We warn that several definitions coexist in the literature. In this case, it is worth noting that by \cite[Theorem 7.2.40 p.~244]{Cho19} the topological dual $V^*$ admits a canonical structure of order unit space (we refer to \cite[Definition 7.2.17 p.~235]{Cho19} for the definition). In particular, $V^*$ is an ordered real vector space by the dual cone $C' \ov{\mathrm{def}}{=} \{w \in V^* : \la w,x \ra \geq 0 \text{ for all } x \in C \}$. Finally, note that $u$ is continuous by \cite[Corollary 2.33 p.~83]{AlT07}. 

\paragraph{Generalized probabilistic theories}
In this paper, a generalized probabilistic theory is a Banach base-norm space $(V,C,u)$. This concept was introduced in \cite[p.~242]{DL70} in a form that is essentially equivalent and under the name state space. We also refer to \cite[p.~72]{Lam18} for the finite-dimensional case. The base defined in \eqref{base-eq} is called the state space of the theory and is denoted by
\[
\Omega 
\ov{\mathrm{def}}{=} \{ \omega \in C : \la u,\omega \ra = 1 \}.
\]
The element $u$ is called the unit effect. It can be interpreted as a measurement with a single outcome, which occurs with certainty for every state. An effect is a functional $e \in V^*$ such that $0 \leq e \leq u$. Given a state $\omega \in \Omega$ and an effect $e$, the scalar $\la e,\omega \ra \in [0,1]$ can be seen as the probability of the outcome corresponding to $e$ when performing the associated measurement on the system.


Following \cite[p.~37]{AlS03}, the set $\cal{K}$ of normal states on a $\JBW$-algebra $\A$ is called the normal state space of $\A$. We have the following result, which generalizes \cite[Theorem 2.3.1 p.~106]{Fur16} which treats the particular case of selfadjoint parts of von Neumann algebras. Here $\A_*$ denotes the predual of the $\JBW$-algebra $\A$. We equally use the contractive linear form $u \co \A_* \to \R$, $\phi \mapsto \phi(1)$.

\begin{prop}
\label{prop-Jordan-GPT}
Let $\A$ be a $\JBW$-algebra. Then $(\A_*,\A_{*+},u)$ defines a generalized probabilistic theory. Moreover, its state space $
\Omega$ is exactly the normal state space $\cal{K}$ of $\A$.
\end{prop}

\begin{proof}
According \cite[Corollary 2.60 p.~67]{AlS03}, $\A_*$ is a base-norm space in the Alfsen-Shultz sense. Moreover, $\A_*$ is complete and its cone $\A_{*+}$ is closed. Using \cite[Proposition 2.2.20 p.~103]{Fur16}, we conclude that $(\A_*,\A_{*+},u)$ is a Banach base-norm space.
\end{proof}

\paragraph{Entropy}
Now, suppose that the $\JBW$-algebra $\A$ admits a normal finite faithful trace $\tau$.  Following \cite[p.~137]{Ioc84}, we define $\L^1(\A)$ as the completion of $\A$ for the norm $\norm{\cdot}_{\L^1(\cal{M})}$ defined in \eqref{norm-LpNA}. We also consider the closure $\L^1(\A)_+$ of $\A_+$. According to \cite[Theorem 2.4 p.~479]{AyA85} \cite[Theorem V.2.2 p.~138]{Ioc84}, for any positive normal linear functional $\phi$ on $\A$ there exists an element $h \in \L^1(\A)_+$ such that $\phi(x)=\tau(h \circ x)$ for any $x \in \A$. Conversely, any element $h$ from $\L^1(\A)_+$ defines a positive normal linear functional $\phi$ on $\A$ by $\phi(x)=\tau(h \circ x)$. We say that $h$ is the density of the state $\phi$, and we denote it by $d_\phi$. 

Using a concrete realization of the space $\L^1(\A)$ described in \cite{Ayu84}, together with an appropriate spectral decomposition, one could define the entropy of an arbitrary element of $\L^1(\A)_+$ with trace 1. However, this lies beyond the scope of the present paper. Here, we restrict ourselves to defining the entropy for elements of $\A_+$ with trace 1 using the 
spectral theory of \cite[Section 4]{ASS78} and \cite[Section 3.2]{HOS84}, and leave the general case for future investigations. We need to consider the function $g \co \R^+ \to \R$ defined by $g(0) \ov{\mathrm{def}}{=} 0$ and $g(t) \ov{\mathrm{def}}{=} t\log_2(t)$ if $t > 0$. By a slight abuse, if $x \in \A_+$ we use the notation $x\log_2 x \ov{\mathrm{def}}{=} g(x)$.
\begin{defi}
\label{def-entropy}
Let $\A$ be a $\JBW$-algebra equipped with a normal finite faithful trace $\tau$. For any element $x \in \A_+$ with $\tau(x)=1$, we define the entropy
$$
\H_\tau(x)
\ov{\mathrm{def}}{=} -\tau (x \log_2x).
$$
The entropy of a normal state $\phi$ is the entropy $\H_\tau(d_\phi)$ of its density $d_\phi$.
\end{defi}

\section{Open questions}
\label{sec-open-question}

We finish with some open questions.

\begin{quest}
\label{octonions}
Let $\Omega$ be a (localizable) measure space. Is it true that \eqref{norm-LpNA} defines a norm in the case $\cal{M}=\L^\infty(\Omega,\H_3(\O_{\mathbb{C}}))$?
\end{quest}

This question is open in the particular case $\cal{M}=\H_3(\O_{\mathbb{C}})$. Following Remark \ref{Rodrigues}, the following question is natural. Here $s_1(x),s_2(x),s_3(x)$ are the singular numbers of the element $x$.

\begin{quest}
\label{quest-octonions-link}
Suppose that $1 \leq p < \infty$. Consider some element $x \in \H_3(\O_\mathbb{C})$. What is the relation between $\big(\sum_{i=1}^3 s_i(x)^p\big)^{\frac{1}{p}}$ and $\norm{x}_{\L^p(\H_3(\O_\mathbb{C}))}$ and $\norm{x}_{\L^{p,A}(\H_3(\O_\mathbb{C}))}$?
\end{quest}


\paragraph{Competing interests} The authors declare that they have no competing interests.

\paragraph{Data availability} No data sets were generated during this study.

\paragraph{Acknowledgment} 
Part of this work was carried out during a visit of the first author to Nankai University (China) in 2025, supported by the National Natural Science Foundation of China (Grant No.~12571143). The first author gratefully acknowledges the hospitality of Nankai University.




{\footnotesize

\vspace{0.2cm}

\noindent C\'edric Arhancet\\ 
\noindent 6 rue Didier Daurat, 81000 Albi, France\\
URL: \href{http://sites.google.com/site/cedricarhancet}{https://sites.google.com/site/cedricarhancet}\\
cedric.arhancet@protonmail.com\\
ORCID: 0000-0002-5179-6972 

\vspace{0.2cm}

\noindent Lei Li \\
\noindent School of Mathematical Sciences and LPMC \\
Nankai University, 300071 Tianjin, China \\
leilee@nankai.edu.cn

}

\end{document}